\documentclass[11pt]{article}

\usepackage[T1]{fontenc}
\usepackage[utf8]{inputenc}
\usepackage{lmodern}
\usepackage{microtype}
\usepackage[margin=1.1in]{geometry}

\usepackage{amsmath,amssymb,amsfonts,amsthm,mathtools}
\usepackage{bm}
\usepackage{enumitem}
\usepackage{booktabs}
\usepackage{array}
\usepackage{xcolor}

\usepackage[colorlinks=true,linkcolor=blue!45!black,citecolor=blue!45!black,urlcolor=blue!45!black]{hyperref}

\makeatletter
\@ifundefined{theorem}{\newtheorem{theorem}{Theorem}[section]}{}
\@ifundefined{proposition}{\newtheorem{proposition}[theorem]{Proposition}}{}
\@ifundefined{lemma}{\newtheorem{lemma}[theorem]{Lemma}}{}
\@ifundefined{corollary}{\newtheorem{corollary}[theorem]{Corollary}}{}
\@ifundefined{definition}{\newtheorem{definition}[theorem]{Definition}}{}
\@ifundefined{assumption}{}{}
\@ifundefined{remark}{\newtheorem{remark}[theorem]{Remark}}{}
\@ifundefined{convention}{}{}
\@ifundefined{notation}{}{}
\makeatother

\providecommand{\R}{\mathbb{R}}
\providecommand{\C}{\mathbb{C}}
\providecommand{\Z}{\mathbb{Z}}
\providecommand{\N}{\mathbb{N}}

\providecommand{\ii}{\mathrm{i}}
\providecommand{\ee}{\mathrm{e}}
\providecommand{\dd}{\,\mathrm{d}}
\providecommand{\eps}{\varepsilon}
\providecommand{\hbarc}{\hbar c}

\providecommand{\E}{\mathbb{E}}
\providecommand{\Tr}{\operatorname{Tr}}
\providecommand{\Trphys}{\operatorname{Tr}_{\mathrm{phys}}}
\providecommand{\FP}{\operatorname{FP}}

\providecommand{\Dom}{\operatorname{Dom}}

\providecommand{\Ker}{\operatorname{Ker}}

\providecommand{\curl}{\operatorname{curl}}

\providecommand{\abs}[1]{\left\lvert #1\right\rvert}
\providecommand{\norm}[1]{\left\lVert #1\right\rVert}
\providecommand{\inner}[2]{\left\langle #1,#2\right\rangle}

\providecommand{\set}[1]{\left\{#1\right\}}
\providecommand{\paren}[1]{\left(#1\right)}
\providecommand{\brac}[1]{\left[#1\right]}

\providecommand{\OmegaLa}{\Omega_{L,a}}
\providecommand{\TtwoL}{T_L^2}
\providecommand{\Aarea}{A}
\providecommand{\Hphys}{\mathcal H_{\mathrm{phys}}}
\providecommand{\HPEC}{\mathcal H_{\mathrm{PEC}}}
\providecommand{\VPEC}{\mathcal V_{\mathrm{PEC}}}
\providecommand{\Vphys}{\mathcal V_{\mathrm{phys}}}
\providecommand{\Hcurl}{H(\operatorname{curl})}
\providecommand{\Hzcurl}{H_0(\operatorname{curl})}

\providecommand{\hzero}{h_0}
\providecommand{\Pzero}{P_0}

\providecommand{\qPEC}{q_{\mathrm{PEC}}}
\providecommand{\qMx}{q_{\mathrm{Mx}}}
\providecommand{\LPEC}{\mathcal L_{\mathrm{PEC}}}
\providecommand{\LMx}{\mathcal L_{\mathrm{Mx}}}
\providecommand{\Lamb}{\mathcal L_{\mathcal M}}
\providecommand{\VMx}{\mathcal V_{\mathrm{Mx}}}
\providecommand{\LD}{L_{\mathrm D}}
\providecommand{\LN}{L_{\mathrm N}}
\providecommand{\LNp}{L_{\mathrm N}^{\prime}}

\providecommand{\lambdamn}{\lambda_{m,n}}
\providecommand{\phim}{\phi_m}
\providecommand{\sn}{s_n}
\providecommand{\cn}{c_n}
\providecommand{\khat}{\widehat{k}}
\providecommand{\khatperp}{\widehat{k}^{\perp}}

\providecommand{\XiNoise}{\Xi}
\providecommand{\Sigmatau}{\Sigma_{\tau}}
\providecommand{\Ctau}{C_{\tau}}
\providecommand{\UtauMx}{U_{\tau}^{\mathrm{Mx}}}
\providecommand{\EDtau}{E_{D,\tau}}
\providecommand{\ENtau}{E_{N,\tau}^{\prime}}
\providecommand{\Btau}{B_{\tau}}

\numberwithin{equation}{section}
\allowdisplaybreaks

\setlist[itemize]{leftmargin=2.2em}
\setlist[enumerate]{leftmargin=2.2em}

\title{A Maxwell Quadratic-Form Representation of the Parallel-Plate Casimir Trace from Codimension-Three Riesz Reduction}
\author{
Irshadullah Khan\thanks{\url{https://www.researchgate.net/profile/Irshadullah_Khan/research}}\\
Department of Mathematics\\
Quaid-i-Azam University\thanks{Visiting Faculty under the UNDP TOKTEN programme.}\\
Islamabad, Pakistan\\
\texttt{irshadk2@gmail.com}
\and
\href{https://orcid.org/0000-0001-8823-6752}{Bilal Khan}\thanks{\url{https://engineering.lehigh.edu/faculty/bilal-khan}}\\
Department of Computer Science\\
Lehigh University\\
Bethlehem, PA, USA\\
\texttt{bik221@lehigh.edu}
}
\date{}

\begin{document}

\maketitle

\begin{abstract}
We formulate a Maxwell version of the codimension-three Riesz/Gaussian
quadratic-form representation for perfectly conducting parallel plates.  This
paper is the Maxwell follow-up to the scalar codimension-three
Riesz/Gaussian representation theorem of \cite{KhanKhan2026Scalar}: the same
transverse Riesz reduction and prescribed-covariance quadratic-form mechanism
are carried over here to the physical parallel-plate Maxwell operator.  The
construction is carried out in finite lateral volume
\(\Omega_{L,a}=T_L^2\times[0,a]\), using the physical electric-field Hilbert
space of divergence-free fields satisfying the perfect-conductor tangential
condition \(n\times E=0\), with the static normal zero mode removed.  The
Maxwell curl--curl operator is defined by its closed quadratic form, and an
explicit Fourier-domain analysis proves the finite-volume spectral gap,
compact resolvent, and heat-trace admissibility needed for the stochastic
construction.

For this reduced Maxwell operator \(\mathcal L_{\mathrm{Mx}}\), the
codimension-three Riesz integral gives the transversely reduced Riesz mediator
\(g\mathcal L_{\mathrm{Mx}}^{-1}\).  A prescribed heat-regularized Gaussian
source with covariance
\((\hbar c/g)\mathcal L_{\mathrm{Mx}}^{3/2}e^{-\tau\mathcal L_{\mathrm{Mx}}}\)
then has expected quadratic Green energy equal to the heat-regularized physical
Maxwell trace.  The finite-volume trace is shown to be spectrally equivalent
to a scalar Dirichlet channel plus a scalar Neumann channel with its constant
zero mode removed.  Under the standard parallel-plate interaction finite-part
prescription, the large-area energy density is
\[
    -\frac{\pi^2\hbar c}{720a^3}.
\]
The result is a representation theorem for the Maxwell parallel-plate trace
under a prescribed covariance in the flat-plate geometry considered here.
\end{abstract}

\tableofcontents

\section{Introduction, scope, and main theorem}
\label{sec:introduction-scope-theorem}

The scalar Riesz/Gaussian construction represents a heat-regularized scalar
Casimir trace as the expectation of a quadratic form.  Its operator-theoretic
core is simple: in codimension three, transverse reduction of the fractional
product operator \((L+\abs q^2)^{-5/2}\) produces the Green multiplier
\(L^{-1}\), up to the constant \((6\pi^2)^{-1}\).  If a heat-regularized
Gaussian source is then assigned covariance proportional to
\(L^{3/2}e^{-\tau L}\), pairing this covariance with the induced Green
operator produces the trace multiplier \(L^{1/2}e^{-\tau L}\).
The companion scalar preprint \cite{KhanKhan2026Scalar} established this
mechanism for positive scalar operators and, in the Dirichlet parallel-plate
benchmark, recovered the standard one-channel scalar interaction finite part.
The present paper asks whether the same representation survives for the true
Maxwell plate operator, where divergence constraints, boundary traces, zero
modes, and physical polarizations must all be handled explicitly.

The purpose of the present paper is to carry out the first Maxwell extension of
that representation in the one geometry where all boundary and polarization
issues can be handled explicitly: two perfectly conducting parallel plates.  We
replace the scalar plate operator by the physical Maxwell curl--curl operator
on divergence-free electric fields satisfying the perfect-conductor tangential
boundary condition.  The construction is performed first in finite lateral
volume,
\[
    \OmegaLa=\TtwoL\times[0,a],
    \qquad \Aarea=L^2,
\]
with periodic lateral directions and conducting plates at \(z=0\) and
\(z=a\).  The large-area limit is taken only after the heat-regularized trace
density has been formed.  We use the term \emph{transverse reduction} for the
codimension-three operator-valued integral; no additional geometric brane
dynamics is involved.

There are two reasons for this finite-volume formulation.  First, the Maxwell
operator has a static normal zero mode, which must be removed before a reduced
inverse can be used.  Second, after that removal the finite-volume operator has
a positive spectral gap, whereas the gap collapses as \(L\to\infty\) because of
the zero-vertical-mode branch.  Thus the inverse-operator identities below are
finite-volume statements.  The infinite-area result is an interaction
finite-part statement for traces per unit area.

\subsection{Scope}
\label{subsec:scope}

The paper proves a representation theorem for the reduced physical Maxwell
operator in the parallel-plate geometry.  It does not derive the electromagnetic
field from a microscopic quantum field theory, and it does not derive the
Gaussian covariance from a photon path integral, a matter loop, or a QED
effective action.  The covariance is prescribed, just as in the scalar
representation theorem, and the result records the quadratic-form identity that
follows from that prescribed covariance.
The overlap with \cite{KhanKhan2026Scalar} is therefore the operator-theoretic
Riesz/Gaussian mechanism and the use of a prescribed heat-regularized
covariance.  The extension proved here is the passage from that scalar
framework to the reduced physical Maxwell operator with perfect-conductor
boundary conditions.  The plate-scale calibration and cube-extremality
discussion from the scalar preprint are not used in the present paper.

The ambient product operator used below is also a mathematical Riesz-type
construction.  It is not asserted to be the ordinary Green operator of a local
six-dimensional Maxwell theory.  Its role is to provide a codimension-three
spectral reduction whose multiplier after transverse reduction is the Maxwell Green
operator \(\LMx^{-1}\) on the reduced physical Hilbert space.

The paper also does not address non-planar conductors, spherical shells, cube
reference-cell calibration, loop corrections, running couplings, or any
fundamental-constant interpretation.  The single physical benchmark evaluated
here is the standard parallel-plate interaction finite part for perfect
conductors \cite{Casimir1948,Bordag2009}.

\subsection{Strategy}
\label{subsec:introduction-strategy}

The proof is organized around five gates.
The first four gates parallel the scalar construction of
\cite{KhanKhan2026Scalar} at the level of spectral calculus and prescribed
covariance.  The genuinely new gate is the fifth: the Maxwell trace must be
identified from the physical mode structure itself, not by any scalar doubling
shortcut.

First, the Maxwell operator is defined without using a vector potential.  The
primary Hilbert-space variable is the electric field \(E\).  Perfect-conductor
boundary conditions are imposed as
\[
    n\times E=0
    \qquad\text{on the plates},
\]
and the physical transverse condition is
\[
    \nabla\cdot E=0 .
\]
The form domain is
\[
    \VPEC
    =
    \set{E\in\Hzcurl(\OmegaLa;\C^3):
          \nabla\cdot E=0\text{ in }\mathcal D'(\OmegaLa)} .
\]
The unreduced Hilbert space is \(\HPEC=\overline{\VPEC}^{\,L^2}\).  The
unique static normal zero mode is
\[
    \hzero=(L^2a)^{-1/2}e_z,
\]
and the physical Hilbert space is
\[
    \Hphys=\HPEC\ominus\operatorname{span}\{\hzero\} .
\]
The reduced Maxwell operator \(\LMx\) is the positive self-adjoint operator
associated with
\[
    \qMx[E,F]
    =
    \int_{\OmegaLa}
    (\nabla\times E)\cdot\overline{(\nabla\times F)}\,\dd x,
    \qquad E,F\in\Vphys .
\]
This form-first definition prevents ambiguity about the operator domain.

Second, the form domain is characterized by an explicit mixed sine/cosine
Fourier expansion.  The tangential components use sine modes in the plate
coordinate and the normal component uses cosine modes.  The divergence-free
condition becomes the coefficient constraint
\[
    \ii k_m\cdot a_{m,n}-\nu_n b_{m,n}=0
    \qquad(n\geq1),
\]
and the reduced physical space is obtained by imposing \(b_{0,0}=0\).  This
coefficient-domain lemma ties the subsequent mode calculation to the actual
closed Maxwell form domain.

Third, with the reduced finite-volume operator in hand, the codimension-three
Riesz reduction is applied by spectral calculus:
\[
    \int_{\R^3}\frac{\dd^3q}{(2\pi)^3}
    (\LMx+\abs q^2)^{-5/2}
    =
    \frac{1}{6\pi^2}\LMx^{-1} .
\]
Thus, with \(g=\kappa/(6\pi^2)\), the transversely reduced Maxwell Riesz
mediator is \( \VMx=g\LMx^{-1}\).
Fourth, a heat-regularized Gaussian source is prescribed on \(\Hphys\):
\[
    \Sigmatau
    =
    \paren{\frac{\hbarc}{g}}^{1/2}
    \LMx^{3/4}e^{-\tau\LMx/2}\XiNoise .
\]
Its covariance is
\[
    \Ctau
    =
    \frac{\hbarc}{g}\LMx^{3/2}e^{-\tau\LMx} .
\]
The heat-trace estimates proved from the finite-volume spectrum imply that
\(\Sigmatau\) is an honest \(\Hphys\)-valued Gaussian random variable for every
\(\tau>0\).  Consequently,
\[
    \E\left[
        \frac12\inner{\Sigmatau}{g\LMx^{-1}\Sigmatau}_{\Hphys}
    \right]
    =
    \frac{\hbarc}{2}
    \Trphys\paren{\LMx^{1/2}e^{-\tau\LMx}} .
\]

Fifth, the physical Maxwell trace is evaluated.  The TE/TM analysis is not
implemented by inserting a scalar factor of two.  Instead one proves the
spectral equivalence
\[
    \LMx\simeq \LD\oplus\LNp,
\]
where \(\LD\) is the scalar Dirichlet Laplacian and \(\LNp\) is the scalar
Neumann Laplacian with the constant zero mode removed.  Thus
\[
    \Trphys f(\LMx)
    =
    \Tr f(\LD)+\Tr' f(\LN)
\]
for Borel functions \(f\) for which these traces are finite.  The remaining
zero-vertical Neumann branch is
retained in the regulated trace and removed only at the interaction
finite-part stage because it is independent of the plate separation \(a\).

\subsection{Main theorem}
\label{subsec:main-theorem}

The following theorem is the result assembled by the technical sections.

\begin{theorem}[Maxwell parallel-plate Riesz/Gaussian representation]
\label{thm:main-maxwell-representation}
Let
\[
    \OmegaLa=\TtwoL\times[0,a]
\]
with periodic lateral directions and perfectly conducting plates at
\(z=0,a\).  Let \(\LMx\) be the reduced perfect-conductor Maxwell operator on
\(\Hphys\), defined by the closed form
\[
    \qMx[E,F]
    =
    \int_{\OmegaLa}
    (\nabla\times E)\cdot\overline{(\nabla\times F)}\,\dd x,
    \qquad E,F\in\Vphys,
\]
where \(\Vphys\) consists of divergence-free electric fields satisfying
\(n\times E=0\) on the plates, after removing the static normal mode
\(\hzero=(L^2a)^{-1/2}e_z\).  Then, for finite \(L\), \(\LMx\) is positive
self-adjoint on \(\Hphys\), has compact resolvent, and satisfies
\[
    \LMx\geq
    \ell_{L,a}I,
    \qquad
    \ell_{L,a}:=
    \min\set{\paren{\frac{2\pi}{L}}^2,
              \paren{\frac{\pi}{a}}^2}.
\]
Moreover,
\[
    \Tr\paren{\LMx^p e^{-\tau\LMx}}<\infty
    \qquad(\tau>0,\ p\geq0).
\]

Define
\[
    \VMx
    :=
    \kappa
    \int_{\R^3}\frac{\dd^3q}{(2\pi)^3}
    (\LMx+\abs q^2)^{-5/2},
    \qquad
    g:=\frac{\kappa}{6\pi^2} .
\]
Then
\[
    \VMx=g\LMx^{-1}
\]
as a bounded operator on \(\Hphys\) for finite \(L\).

Let \(\XiNoise\) be cylindrical white noise on \(\Hphys\).  For \(\tau>0\), set
\[
    \Sigmatau
    :=
    \paren{\frac{\hbarc}{g}}^{1/2}
    \LMx^{3/4}e^{-\tau\LMx/2}\XiNoise .
\]
Then \(\Sigmatau\) is an \(\Hphys\)-valued Gaussian random variable with
covariance
\[
    \Ctau
    =
    \frac{\hbarc}{g}\LMx^{3/2}e^{-\tau\LMx} .
\]
For
\[
    \UtauMx
    :=
    \frac12\inner{\Sigmatau}{g\LMx^{-1}\Sigmatau}_{\Hphys},
\]
one has the regulated identity
\[
    \E[\UtauMx]
    =
    \frac{\hbarc}{2}
    \Trphys\paren{\LMx^{1/2}e^{-\tau\LMx}} .
\]
In the large-area limit below, \(\UtauMx=\UtauMx(L,a)\) denotes the
finite-volume quadratic energy associated with \(\Omega_{L,a}\).
Furthermore, the Maxwell trace decomposes as
\[
    \Trphys f(\LMx)=\Tr f(\LD)+\Tr' f(\LN)
    =\Tr f(\LD)+\Tr f(\LNp)
\]
for every Borel function \(f\) for which these traces are finite.  Here
\(\LD\) is the scalar Dirichlet Laplacian, \(\LN\) is the scalar Neumann
Laplacian, and \(\LNp\) is the restriction of \(\LN\) to the orthogonal
complement of its constant zero mode.  The prime on \(\Tr' f(\LN)\) removes
only the constant Neumann zero mode.  The remaining \(n=0,\ m\ne0\) Neumann
branch is retained in the finite-volume trace and contributes an
\(a\)-independent term to the large-area density.  Consequently, under the standard
parallel-plate interaction finite-part prescription,
\[
    \boxed{
    \FP_{\mathrm{int},\,\tau\to0^+}
    \left[
        \lim_{L\to\infty}\frac{1}{L^2}\E[\UtauMx(L,a)]
    \right]
    =
    -\frac{\pi^2\hbarc}{720a^3} .
    }
\]
\end{theorem}

\begin{proof}[Proof roadmap]
Section~\ref{sec:maxwell-operator} defines the physical electric-field Hilbert
space and the reduced Maxwell operator by a closed quadratic form.
Section~\ref{sec:fourier-domain-heat-trace} proves the mixed Fourier
coefficient-domain lemma, identifies the zero mode, proves the finite-volume
spectral gap, and proves the heat-trace estimates.  Section~\ref{sec:riesz-gaussian-source}
proves the Hilbert-space Riesz reduction, constructs the mollified
transversely reduced Maxwell Riesz mediator, defines the heat-regularized Gaussian
source, and proves the regulated quadratic-form trace identity.
Section~\ref{sec:te-tm-decomposition} proves the spectral equivalence
\(\LMx\simeq\LD\oplus\LNp\).  Finally,
Section~\ref{sec:finite-part-plate-result} defines the standard interaction
finite part, removes the \(a\)-independent zero-vertical branch, and evaluates
the Dirichlet and reduced Neumann contributions.  The final displayed formula
is Theorem~\ref{thm:Maxwell-parallel-plate-finite-part}.
\end{proof}

\begin{remark}[Finite-volume discipline]
\label{rem:intro-finite-volume-discipline}
The theorem should be read in the order stated.  The identity
\(\VMx=g\LMx^{-1}\) uses the finite-volume spectral gap.  The large-area limit
is taken only after forming the heat-regularized trace density and applying the
standard interaction finite-part prescription.  No bounded inverse
\(\LMx^{-1}\) is asserted in the infinite-area limit.
\end{remark}

\begin{remark}[Prescribed covariance]
\label{rem:intro-prescribed-covariance}
The covariance
\[
    \Ctau=\frac{\hbarc}{g}\LMx^{3/2}e^{-\tau\LMx}
\]
is part of the data of the representation.  The paper proves the exact
quadratic-form identity that follows from this covariance.  It does not derive
this covariance from microscopic quantum electrodynamics.
\end{remark}
\section{The reduced Maxwell operator for perfectly conducting plates}
\label{sec:maxwell-operator}

This section defines the Maxwell operator used in the representation theorem.  The construction is deliberately finite-volume.  The slab is
\[
    \OmegaLa:=\TtwoL\times[0,a],
    \qquad
    \Aarea=L^2,
\]
with periodic lateral variables and two perfectly conducting plates at
\(z=0\) and \(z=a\).  We work over complex vector fields; the real theory is the corresponding real subspace.

The primary field variable is the electric field \(E\), not a vector potential.  This avoids any need to choose a gauge in the main proof.  Gauge-fixed vector-potential formulations may be compared with this construction later, but the operator below is defined directly on the physical transverse electric-field space.

\subsection{Perfect-conductor boundary conditions}
\label{subsec:pec-boundary}

For a perfect electric conductor, the physical boundary conditions for the electromagnetic fields are
\[
    n\times E=0,
    \qquad
    n\cdot B=0,
\]
on the conducting boundary.  For a nonzero time-harmonic source-free mode with time dependence \(e^{-\ii\omega t}\), Maxwell's equations give
\[
    B=\frac{1}{\ii\omega}\nabla\times E
\]
up to the conventional placement of the factor \(c\), and the electric field satisfies
\[
    \nabla\times\nabla\times E=\frac{\omega^2}{c^2}E,
    \qquad
    \nabla\cdot E=0 .
\]
Thus the spectral parameter of the positive operator below is \(\lambda=\omega^2/c^2\), and the zero-point energy contribution is \((\hbarc/2)\lambda^{1/2}\).

On the flat plates of \(\OmegaLa\), the outward unit normal is \(n=\pm e_z\).  Hence
\[
    n\times E=0
    \qquad\Longleftrightarrow\qquad
    E_x=E_y=0
    \quad\text{on }z=0,a .
\]
For smooth nonzero modes, this tangential electric condition also implies the normal magnetic condition on the flat plates.  Indeed,
\[
    n\cdot(\nabla\times E)=\partial_xE_y-\partial_yE_x
\]
on a plate.  Since the tangential trace \((E_x,E_y)\) vanishes there, its tangential derivatives vanish in the trace sense, and therefore \(n\cdot B=0\).  The form construction below takes \(n\times E=0\) as the essential boundary condition and imposes transversality by \(\nabla\cdot E=0\).

Equivalently, if \(E\) is identified with the one-form
\[
    E_x\,dx+E_y\,dy+E_z\,dz,
\]
then \(n\times E=0\) on the plates is the condition that the pullback of this
one-form to the boundary vanish.  In differential-form terminology this is the
relative trace condition for one-forms, namely the form-domain part of the
relative boundary condition.  We will not rely on that language in the proofs.

\subsection{The PEC form domain}
\label{subsec:pec-domain}

Let
\[
    \Hcurl(\OmegaLa;\C^3)
    :=
    \set{E\in L^2(\OmegaLa;\C^3):\nabla\times E\in L^2(\OmegaLa;\C^3)}
\]
with derivatives understood in the sense of distributions and with graph norm
\[
    \norm{E}_{\Hcurl}^2
    :=
    \norm{E}_{L^2}^2+\norm{\nabla\times E}_{L^2}^2 .
\]
Because the lateral variables live on the torus \(\TtwoL\), there are no lateral boundary faces.  We use the following convention throughout the paper:
\[
    \Hzcurl(\OmegaLa;\C^3)
\]
denotes the closure in the \(\Hcurl\)-graph norm of smooth lateral-periodic vector fields whose tangential components satisfy
\[
    E_x=E_y=0
    \quad\text{at }z=0,a .
\]
Thus \(\Hzcurl\) encodes zero tangential electric trace on the two conducting plates.
This is the standard \(H(\curl)\) framework for Maxwell boundary-value
problems with perfect-conductor trace conditions; see, for example,
\cite{Monk2003}.

Define the divergence-free perfect-conductor form domain by
\begin{equation}
\label{eq:VPEC-def}
    \VPEC
    :=
    \set{
        E\in \Hzcurl(\OmegaLa;\C^3):
        \nabla\cdot E=0\text{ in }\mathcal D'(\OmegaLa)
    } .
\end{equation}
Here \(\mathcal D'(\OmegaLa)\) is understood with lateral periodicity.  The divergence constraint is closed under \(L^2\)-convergence: if \(E_j\to E\) in \(L^2\) and \(\nabla\cdot E_j=0\), then for every smooth periodic test function \(\varphi\) compactly supported in the open interval \((0,a)\) in the plate variable,
\[
    \inner{\nabla\cdot E}{\varphi}
    =
    -\int_{\OmegaLa} E\cdot\nabla\overline{\varphi}\,\dd x
    =
    \lim_{j\to\infty}
    -\int_{\OmegaLa} E_j\cdot\nabla\overline{\varphi}\,\dd x
    =0 .
\]
Consequently \(\VPEC\) is a closed subspace of \(\Hzcurl(\OmegaLa;\C^3)\) in the \(\Hcurl\)-graph norm.

The unreduced PEC Hilbert space is the \(L^2\)-closure of this form domain:
\begin{equation}
\label{eq:HPEC-def}
    \HPEC:=\overline{\VPEC}^{\,L^2(\OmegaLa;\C^3)} .
\end{equation}
The boundary condition is imposed at the level of the form domain \(\VPEC\), not as a pointwise condition on arbitrary elements of \(\HPEC\).

\subsection{The static normal mode and the reduced physical space}
\label{subsec:static-mode-reduction}

The constant normal field
\[
    e_z=(0,0,1)
\]
satisfies
\[
    \nabla\times e_z=0,
    \qquad
    \nabla\cdot e_z=0,
    \qquad
    n\times e_z=0\text{ on }z=0,a .
\]
Hence it belongs to \(\VPEC\) and has zero Maxwell energy.  Its normalized version is
\begin{equation}
\label{eq:hzero-def}
    \hzero:=\abs{\OmegaLa}^{-1/2}e_z=(L^2a)^{-1/2}e_z .
\end{equation}
This is the uniform static electric field normal to the plates.  It represents a zero-frequency electrostatic sector, not a photon oscillator contributing to the Casimir trace.

We therefore remove this mode before defining the physical Maxwell operator.  Let
\begin{equation}
\label{eq:Pzero-def}
    \Pzero E:=\inner{E}{\hzero}_{L^2(\OmegaLa)}\hzero
\end{equation}
be the orthogonal projection onto \(\operatorname{span}\{\hzero\}\).  Define
\begin{equation}
\label{eq:Hphys-def}
    \Hphys:=\HPEC\ominus\operatorname{span}\{\hzero\}
    = (I-\Pzero)\HPEC
\end{equation}
and
\begin{equation}
\label{eq:Vphys-def}
    \Vphys:=\VPEC\cap\Hphys .
\end{equation}
The space \(\Vphys\) is dense in \(\Hphys\).  Indeed, if \(F\in\Hphys\), choose \(E_j\in\VPEC\) with \(E_j\to F\) in \(L^2\).  Since \(\hzero\in\VPEC\), each
\[
    E_j-\Pzero E_j
\]
belongs to \(\VPEC\cap\Hphys=\Vphys\), and
\[
    E_j-\Pzero E_j\longrightarrow F-\Pzero F=F
\]
in \(L^2\).

The next section proves, by an explicit mixed Fourier decomposition, that \(\operatorname{span}\{\hzero\}\) is the entire kernel of the unreduced finite-volume PEC Maxwell form.  This section only removes the known static normal mode and constructs the reduced operator.  Positivity with a finite-volume spectral gap is not asserted until after the Fourier-domain analysis.

\subsection{The closed Maxwell quadratic form}
\label{subsec:maxwell-form}

Define the unreduced perfect-conductor Maxwell form on \(\HPEC\) by
\begin{equation}
\label{eq:qPEC-def}
    \qPEC[E,F]
    :=
    \int_{\OmegaLa}
    (\nabla\times E)\cdot\overline{(\nabla\times F)}\,\dd x,
    \qquad
    E,F\in\VPEC .
\end{equation}
The form is nonnegative.  It is also closed: the norm
\[
    \norm{E}_{\qPEC}^2
    :=
    \norm{E}_{L^2}^2+\qPEC[E,E]
    =
    \norm{E}_{L^2}^2+\norm{\nabla\times E}_{L^2}^2
\]
is exactly the \(\Hcurl\)-graph norm on the closed subspace \(\VPEC\).  Since \(\VPEC\) is complete in this norm, \(\qPEC\) is a closed densely defined nonnegative form on \(\HPEC\).

Restrict this form to the reduced physical space:
\begin{equation}
\label{eq:qMx-def}
    \qMx[E,F]
    :=
    \qPEC[E,F]
    =
    \int_{\OmegaLa}
    (\nabla\times E)\cdot\overline{(\nabla\times F)}\,\dd x,
    \qquad
    E,F\in\Vphys .
\end{equation}
The subspace \(\Vphys\) is closed in \(\VPEC\) for the graph norm.  To see this, let \(E_j\in\Vphys\) converge to \(E\in\VPEC\) in the graph norm.  Then \(E_j\to E\) in \(L^2\), and since \(\Hphys\) is closed in \(L^2\), one has \(E\in\Hphys\).  Thus \(E\in\Vphys\).  Hence \(\qMx\) is closed, nonnegative, and densely defined on \(\Hphys\).

By the representation theorem for closed semibounded quadratic forms
\cite{Kato1995}, there is a unique nonnegative self-adjoint operator
\[
    \LMx
\]
on \(\Hphys\) such that
\begin{equation}
\label{eq:form-representation-LMx}
    \qMx[E,F]=\inner{\LMx E}{F}_{\Hphys}
    \qquad
    \text{for all }F\in\Vphys
\end{equation}
whenever \(E\in\Dom(\LMx)\).  Equivalently,
\begin{equation}
\label{eq:LMx-domain-formal}
    \Dom(\LMx)
    =
    \set{
        E\in\Vphys:
        \exists G\in\Hphys\text{ such that }
        \qMx[E,F]=\inner{G}{F}_{\Hphys}
        \text{ for all }F\in\Vphys
    },
\end{equation}
and then \(\LMx E:=G\).

For smooth fields in this operator domain, integration by parts gives the differential expression
\begin{equation}
\label{eq:LMx-curlcurl}
    \LMx E=\nabla\times\nabla\times E .
\end{equation}
Since the form domain imposes \(\nabla\cdot E=0\), this agrees on smooth fields with
\[
    \nabla\times\nabla\times E=-\Delta E .
\]
The boundary condition and transversality are not imposed as additional
pointwise assumptions on \(\Dom(\LMx)\).  They are encoded by the form domain
\(\Vphys\).

\begin{definition}[Reduced PEC Maxwell operator]
\label{def:reduced-pec-maxwell-operator}
The reduced perfect-conductor Maxwell operator on the finite slab \(\OmegaLa\) is the positive self-adjoint operator \(\LMx\) associated with the closed form \(\qMx\) in \eqref{eq:qMx-def} on the Hilbert space \(\Hphys\) in \eqref{eq:Hphys-def}.
\end{definition}

\begin{remark}[Why the electric-field formulation is used]
\label{rem:electric-field-no-gauge}
The construction does not introduce a vector potential and therefore does not require a gauge choice in the main proof.  The physical constraints are imposed directly on the electric field: zero tangential electric trace at the plates, distributional transversality, and removal of the static normal zero mode.  A vector-potential or covariant gauge-fixed formulation may be compared with this operator later, but it is not needed for the Riesz reduction or the stochastic trace identity.
\end{remark}

\begin{remark}[Next spectral step]
\label{rem:maxwell-operator-deferred}
The mixed sine/cosine Fourier description of \(\VPEC\) in the next section
identifies the zero mode and supplies the spectral gap, compact-resolvent, and
heat-trace properties used later in the paper.
\end{remark}
\section{Fourier description and heat-trace admissibility}
\label{sec:fourier-domain-heat-trace}

The preceding section defined \(\LMx\) by a closed quadratic form.  This section ties that abstract form domain to an explicit Fourier model on
\[
    \OmegaLa=\TtwoL\times[0,a].
\]
The result is the finite-volume spectral input needed for the Riesz and stochastic constructions: after removal of the static normal mode, \(\LMx\) has a positive spectral gap, compact resolvent, and finite heat traces
\[
    \Tr\paren{\LMx^p e^{-\tau\LMx}}<\infty
    \qquad(\tau>0,\ p\ge0).
\]
No large-area limit is taken in this section.

\subsection{Mixed Fourier bases}
\label{subsec:mixed-fourier-bases}

Write \(r=(x,y)\in\TtwoL\), and set
\begin{equation}
\label{eq:km-nun-defs}
    k_m:=\frac{2\pi}{L}m,
    \qquad
    m\in\Z^2,
    \qquad
    \nu_n:=\frac{\pi n}{a},
    \qquad
    n\in\N_0 .
\end{equation}
The normalized lateral Fourier modes are
\begin{equation}
\label{eq:phi-m-def}
    \phi_m(r):=\frac1L e^{\ii k_m\cdot r},
    \qquad m\in\Z^2 .
\end{equation}
For the plate variable we use the sine basis
\begin{equation}
\label{eq:sin-basis-def}
    s_n(z):=\sqrt{\frac2a}\sin(\nu_n z),
    \qquad n\ge1,
\end{equation}
and the cosine basis
\begin{equation}
\label{eq:cos-basis-def}
    c_0(z):=a^{-1/2},
    \qquad
    c_n(z):=\sqrt{\frac2a}\cos(\nu_n z),
    \qquad n\ge1 .
\end{equation}
Thus \(\{\phi_m s_n:m\in\Z^2,n\ge1\}\) is an orthonormal basis of \(L^2(\TtwoL\times(0,a))\), and so is \(\{\phi_m c_n:m\in\Z^2,n\ge0\}\).

The choice of sine modes for \(E_x,E_y\) is adapted to the perfect-conductor tangential condition
\[
    E_x=E_y=0\quad\text{on }z=0,a,
\]
whereas no Dirichlet condition is imposed on \(E_z\).  Hence an \(L^2\) vector field is written in the mixed form
\begin{equation}
\label{eq:mixed-expansion-tangential}
    E_{\parallel}(r,z)
    =
    \sum_{m\in\Z^2}\sum_{n\ge1}
    a_{m,n}\,\phi_m(r)s_n(z),
    \qquad
    a_{m,n}\in\C^2,
\end{equation}
where \(E_{\parallel}=(E_x,E_y)\), and
\begin{equation}
\label{eq:mixed-expansion-normal}
    E_z(r,z)
    =
    \sum_{m\in\Z^2}\sum_{n\ge0}
    b_{m,n}\,\phi_m(r)c_n(z),
    \qquad
    b_{m,n}\in\C .
\end{equation}
Parseval gives
\begin{equation}
\label{eq:parseval-L2-mixed}
    \norm{E}_{L^2(\OmegaLa)}^2
    =
    \sum_{m\in\Z^2}\sum_{n\ge1}\abs{a_{m,n}}^2
    +
    \sum_{m\in\Z^2}\sum_{n\ge0}\abs{b_{m,n}}^2 .
\end{equation}
At the level of \(L^2\) alone, the sine expansion should be viewed only as a
choice of orthonormal coordinates.  The perfect-conductor trace condition is
not read pointwise from an arbitrary \(L^2\) series; it is recovered below
from the weighted coefficient domain and the graph-norm closure defining
\(\Hzcurl\).

For later use define
\begin{equation}
\label{eq:lambda-mn-def}
    \lambda_{m,n}:=\abs{k_m}^2+\nu_n^2
    \qquad(n\ge1).
\end{equation}
When \(n=0\), the corresponding normal-branch eigenvalue will be written simply as \(\abs{k_m}^2\).

\subsection{The coefficient-domain lemma}
\label{subsec:coefficient-domain-lemma}

Let \(\mathfrak C_{\mathrm{PEC}}\) be the space of coefficient families
\[
    \set{a_{m,n}\in\C^2\ (m\in\Z^2,n\ge1),
          \ b_{m,n}\in\C\ (m\in\Z^2,n\ge0)}
\]
satisfying
\begin{equation}
\label{eq:CPEC-weighted-summability}
    \sum_{m\in\Z^2}\sum_{n\ge1}
    (1+\lambda_{m,n})
    \paren{\abs{a_{m,n}}^2+\abs{b_{m,n}}^2}
    +
    \sum_{m\in\Z^2}(1+\abs{k_m}^2)\abs{b_{m,0}}^2
    <\infty
\end{equation}
and the coefficient transversality equations
\begin{equation}
\label{eq:coefficient-divergence-constraint}
    \ii k_m\cdot a_{m,n}-\nu_n b_{m,n}=0,
    \qquad
    m\in\Z^2,
    \quad n\ge1 .
\end{equation}
Let
\begin{equation}
\label{eq:Cphys-def}
    \mathfrak C_{\mathrm{phys}}
    :=
    \set{(a,b)\in\mathfrak C_{\mathrm{PEC}}: b_{0,0}=0} .
\end{equation}

\begin{lemma}[Fourier characterization of the PEC form domain]
\label{lem:fourier-characterization-VPEC}
The reconstruction map defined by \eqref{eq:mixed-expansion-tangential}--\eqref{eq:mixed-expansion-normal} identifies \(\mathfrak C_{\mathrm{PEC}}\) with the form domain
\[
    \VPEC
    =
    \set{E\in\Hzcurl(\OmegaLa;\C^3):\nabla\cdot E=0}
\]
defined in \eqref{eq:VPEC-def}.  Under this identification,
\begin{equation}
\label{eq:curl-norm-diagonal}
    \norm{\nabla\times E}_{L^2}^2
    =
    \sum_{m\in\Z^2}\sum_{n\ge1}
    \lambda_{m,n}
    \paren{\abs{a_{m,n}}^2+\abs{b_{m,n}}^2}
    +
    \sum_{m\in\Z^2}\abs{k_m}^2\abs{b_{m,0}}^2 .
\end{equation}
Consequently, \(\mathfrak C_{\mathrm{phys}}\) identifies with
\[
    \Vphys=\VPEC\cap\Hphys .
\]
\end{lemma}

\begin{proof}
We prove both inclusions.  The distributional coefficient calculations and the
density argument are given in full in
Appendix~\ref{app:fourier-domain-lemma}, and we record the main steps here to
keep the spectral calculation self-contained.

First suppose that a coefficient family lies in \(\mathfrak C_{\mathrm{PEC}}\).  Its finite rectangular truncations reconstruct smooth lateral-periodic fields.  Their tangential components are finite sums of sine modes and therefore vanish at \(z=0,a\).  Hence each truncation belongs to the smooth PEC class used to define \(\Hzcurl\).

For a finite truncation, direct differentiation gives
\begin{align}
\label{eq:div-coeff-finite}
    \nabla\cdot E
    &=
    \sum_{m\in\Z^2}\sum_{n\ge1}
    \paren{\ii k_m\cdot a_{m,n}-\nu_n b_{m,n}}
    \phi_m(r)s_n(z),
\end{align}
so the constraint \eqref{eq:coefficient-divergence-constraint} makes the truncation divergence-free.
The curl coefficients are as follows.  For \(n\ge1\),
\begin{align}
\label{eq:curl-coeff-x}
    C^x_{m,n}&:=\ii k_{m,y} b_{m,n}-\nu_n a^y_{m,n},\nonumber\\
    C^y_{m,n}&:=\nu_n a^x_{m,n}-\ii k_{m,x} b_{m,n},\\
    C^z_{m,n}&:=\ii\paren{k_{m,x}a^y_{m,n}-k_{m,y}a^x_{m,n}}.\nonumber
\end{align}
For the normal branch \(n=0\),
\begin{equation}
\label{eq:curl-coeff-zero-branch}
    C^x_{m,0}=\ii k_{m,y}b_{m,0},
    \qquad
    C^y_{m,0}=-\ii k_{m,x}b_{m,0},
    \qquad
    C^z_{m,0}=0 .
\end{equation}
A direct algebraic identity gives, for every \(m\in\Z^2\) and \(n\ge1\),
\begin{align}
\label{eq:curl-div-algebraic-identity}
    &\abs{C^x_{m,n}}^2+
     \abs{C^y_{m,n}}^2+
     \abs{C^z_{m,n}}^2+
     \abs{\ii k_m\cdot a_{m,n}-\nu_n b_{m,n}}^2  \\
    &\hspace{3cm}=
     \lambda_{m,n}
     \paren{\abs{a_{m,n}}^2+\abs{b_{m,n}}^2} .\nonumber
\end{align}
Since the divergence coefficient vanishes, this identity and \eqref{eq:curl-coeff-zero-branch} imply \eqref{eq:curl-norm-diagonal} for finite truncations.  The weighted summability condition \eqref{eq:CPEC-weighted-summability} therefore makes the finite truncations Cauchy in the \(\Hcurl\)-graph norm.  Their graph-norm limit lies in \(\Hzcurl\), the divergence constraint passes to the distributional limit, and the reconstructed field belongs to \(\VPEC\).  This proves
\[
    \mathfrak C_{\mathrm{PEC}}\subset\VPEC .
\]

Conversely, let \(E\in\VPEC\).  Since \(E\in L^2\), it has the mixed coefficient expansion \eqref{eq:mixed-expansion-tangential}--\eqref{eq:mixed-expansion-normal}, with Parseval identity \eqref{eq:parseval-L2-mixed}.  The distributional identity \(\nabla\cdot E=0\) implies the coefficient constraint \eqref{eq:coefficient-divergence-constraint}, because the divergence has the expansion \eqref{eq:div-coeff-finite} in the sense of distributions.  Equivalently, testing against the modes \(\overline{\phi_m}s_n\), \(n\ge1\), yields exactly \eqref{eq:coefficient-divergence-constraint}; no boundary term appears because \(s_n(0)=s_n(a)=0\).

The weak curl \(\nabla\times E\) belongs to \(L^2\).  The coefficient formulas
\eqref{eq:curl-coeff-x} and \eqref{eq:curl-coeff-zero-branch} are understood
distributionally: they are obtained by differentiating finite partial sums,
testing against smooth periodic sine/cosine modes, and passing to the
distributional limit.  Since the weak curl is an \(L^2\) field, uniqueness of
Fourier coefficients identifies its \(L^2\)-Fourier coefficients with the
displayed expressions.  Applying Parseval to \(\nabla\times E\) and using
\eqref{eq:curl-div-algebraic-identity} with the divergence coefficient equal to
zero gives
\[
    \sum_{m\in\Z^2}\sum_{n\ge1}
    \lambda_{m,n}
    \paren{\abs{a_{m,n}}^2+\abs{b_{m,n}}^2}
    +
    \sum_{m\in\Z^2}\abs{k_m}^2\abs{b_{m,0}}^2
    <\infty .
\]
Together with \eqref{eq:parseval-L2-mixed}, this is exactly \eqref{eq:CPEC-weighted-summability}.  Hence the coefficient family of \(E\) lies in \(\mathfrak C_{\mathrm{PEC}}\), and the curl norm is given by \eqref{eq:curl-norm-diagonal}.  Thus
\[
    \VPEC\subset\mathfrak C_{\mathrm{PEC}} .
\]

The two inclusions prove the first assertion.  Finally, the coefficient \(b_{0,0}\) is precisely the component of \(E\) in the direction of the normalized constant normal field \(\hzero=(L^2a)^{-1/2}e_z\).  Therefore imposing \(b_{0,0}=0\) is exactly the condition \(E\perp\hzero\), and \(\mathfrak C_{\mathrm{phys}}\) identifies with \(\VPEC\cap\Hphys=\Vphys\).
\end{proof}

\begin{remark}[Why this lemma is needed]
\label{rem:why-coeff-lemma-needed}
The Maxwell operator was defined by a closed form, not by first writing down separated eigenfunctions.  Lemma~\ref{lem:fourier-characterization-VPEC} proves that the separated coefficient model is the actual form domain.  Thus the spectral calculations below are calculations for the self-adjoint operator \(\LMx\), not merely formal mode counting.
\end{remark}

\subsection{Kernel and reduced inverse in finite volume}
\label{subsec:kernel-and-gap}

The Fourier characterization identifies the zero-energy fields immediately.  Let \(\LPEC\) denote the nonnegative self-adjoint operator associated with the unreduced closed form \(\qPEC\) on \(\HPEC\).

\begin{proposition}[Kernel of the unreduced PEC form]
\label{prop:unreduced-kernel}
The kernel of the unreduced PEC Maxwell form \(\qPEC\) on \(\HPEC\) is
\[
    \Ker \LPEC=\operatorname{span}\{\hzero\} .
\]
Equivalently, the only finite-volume divergence-free PEC field with zero curl is the uniform static normal field.
\end{proposition}

\begin{proof}
If \(E\in\VPEC\) and \(\qPEC[E,E]=0\), then \(\nabla\times E=0\).  By \eqref{eq:curl-norm-diagonal}, all coefficients with positive weights must vanish:
\[
    a_{m,n}=0,
    \quad
    b_{m,n}=0
    \quad(n\ge1),
    \qquad
    b_{m,0}=0
    \quad(m\ne0).
\]
Only \(b_{0,0}\) may remain.  The corresponding field is a constant multiple of \(e_z\), or of the normalized field \(\hzero\).  Conversely, \(e_z\) is divergence-free, satisfies the PEC tangential condition, and has zero curl.  Hence the kernel is exactly \(\operatorname{span}\{\hzero\}\).
\end{proof}

After passing to \(\Hphys=\HPEC\ominus\operatorname{span}\{\hzero\}\), the zero mode is removed.  The first positive spectral scale is
\begin{equation}
\label{eq:finite-gap-ell-La}
    \ell_{L,a}
    :=
    \min\set{
        \paren{\frac{2\pi}{L}}^2,
        \paren{\frac{\pi}{a}}^2
    } .
\end{equation}
The next proposition makes this precise.

\begin{proposition}[Finite-volume spectral gap]
\label{prop:finite-volume-gap}
For finite \(L<\infty\) and \(a>0\), the reduced Maxwell operator satisfies
\begin{equation}
\label{eq:LMx-gap}
    \LMx\ge \ell_{L,a} I
    \qquad\text{on }\Hphys .
\end{equation}
Consequently \(\LMx^{-1}\) is a bounded positive operator on \(\Hphys\), with
\begin{equation}
\label{eq:LMx-inverse-bound}
    \norm{\LMx^{-1}}
    \le
    \ell_{L,a}^{-1} .
\end{equation}
\end{proposition}

\begin{proof}
For \(E\in\Vphys\), the coefficient \(b_{0,0}\) vanishes.  In \eqref{eq:curl-norm-diagonal}, every remaining coefficient is weighted either by \(\lambda_{m,n}\) with \(n\ge1\), or by \(\abs{k_m}^2\) with \(m\ne0\).  The smallest possible value of \(\lambda_{m,n}\) for \(n\ge1\) is \((\pi/a)^2\), attained at \(m=0,n=1\).  The smallest value of \(\abs{k_m}^2\) for \(m\ne0\) is \((2\pi/L)^2\).  Therefore
\[
    \qMx[E,E]
    \ge
    \ell_{L,a}\norm{E}_{L^2}^2
    \qquad(E\in\Vphys),
\]
which is the quadratic-form inequality \eqref{eq:LMx-gap}.  The boundedness of the inverse follows from the spectral theorem.
\end{proof}

\begin{remark}[Why finite volume is essential]
\label{rem:finite-volume-essential}
The bound \eqref{eq:LMx-gap} collapses as \(L\to\infty\), because the normal branch with \(n=0\) has eigenvalues \(\abs{k_m}^2\).  Therefore all uses of \(\LMx^{-1}\) as a bounded operator are finite-volume statements.  The large-area limit is taken later only after heat-regularized traces per unit area have been formed.
\end{remark}

\subsection{Diagonal form and spectrum}
\label{subsec:diagonal-form-spectrum}

The coefficient model also diagonalizes the self-adjoint operator.  For \(n\ge1\), define the constrained block
\begin{equation}
\label{eq:Wmn-block-def}
    W_{m,n}
    :=
    \set{(a,b)\in\C^2\oplus\C:
          \ii k_m\cdot a-\nu_n b=0} .
\end{equation}
For \(n=0\) and \(m\ne0\), define \(W_{m,0}:=\C\), represented by the normal coefficient \(b_{m,0}\).  The removed block \((m,n)=(0,0)\) is absent from the physical space.

For \(n\ge1\), the block \(W_{m,n}\) has dimension two.  Indeed, it is the kernel of one nonzero linear functional on \(\C^3\), because \(\nu_n>0\).  The block \(W_{m,0}\) has dimension one for \(m\ne0\).

Under the unitary coefficient representation of \(\Hphys\), the form \(\qMx\) is
\begin{equation}
\label{eq:qMx-block-diagonal}
    \qMx[E,E]
    =
    \sum_{m\in\Z^2}\sum_{n\ge1}
    \lambda_{m,n}\norm{(a_{m,n},b_{m,n})}_{\C^3}^2
    +
    \sum_{m\in\Z^2\setminus\{0\}}
    \abs{k_m}^2\abs{b_{m,0}}^2 .
\end{equation}
It follows from the representation theorem for diagonal closed forms that \(\LMx\) acts as multiplication by \(\lambda_{m,n}\) on \(W_{m,n}\), \(n\ge1\), and by \(\abs{k_m}^2\) on \(W_{m,0}\), \(m\ne0\).

Thus the reduced Maxwell spectrum is
\begin{equation}
\label{eq:LMx-spectrum-list}
    \lambda_{m,n}
    =
    \abs{k_m}^2+\nu_n^2,
    \qquad
    m\in\Z^2,
    \quad n\ge1,
\end{equation}
with multiplicity two, together with
\begin{equation}
\label{eq:LMx-zero-branch-list}
    \abs{k_m}^2,
    \qquad
    m\in\Z^2\setminus\{0\},
\end{equation}
with multiplicity one.  The missing \(m=0,n=0\) member of the second family is exactly the removed static normal mode.

\begin{proposition}[Trace formula for diagonal spectral functions]
\label{prop:diagonal-trace-formula}
Let \(f:[0,\infty)\to\C\) be Borel and suppose that the sums below converge absolutely.  Then
\begin{equation}
\label{eq:LMx-diagonal-trace-formula}
    \Trphys f(\LMx)
    =
    2\sum_{m\in\Z^2}\sum_{n=1}^{\infty}
    f\paren{\abs{k_m}^2+\nu_n^2}
    +
    \sum_{m\in\Z^2\setminus\{0\}}
    f\paren{\abs{k_m}^2} .
\end{equation}
\end{proposition}

\begin{proof}
The block decomposition above is an orthogonal Hilbert direct sum.  The operator \(\LMx\) is scalar multiplication by \(\lambda_{m,n}\) on each two-dimensional block \(W_{m,n}\), \(n\ge1\), and scalar multiplication by \(\abs{k_m}^2\) on each one-dimensional block \(W_{m,0}\), \(m\ne0\).  Taking the trace of \(f(\LMx)\) block by block gives \eqref{eq:LMx-diagonal-trace-formula}.
\end{proof}

\begin{remark}[Not yet the TE/TM theorem]
\label{rem:not-yet-te-tm}
Formula \eqref{eq:LMx-diagonal-trace-formula} is a diagonal spectral formula.  The later TE/TM section will identify the same trace as a Dirichlet scalar trace plus a reduced Neumann scalar trace.  That later interpretation is not needed for the present heat-trace admissibility result.
\end{remark}

\subsection{Compact resolvent and heat-trace admissibility}
\label{subsec:compact-heat-trace}

\begin{proposition}[Compact resolvent]
\label{prop:compact-resolvent-LMx}
For finite \(L<\infty\) and \(a>0\), the reduced Maxwell operator \(\LMx\) has compact resolvent.
\end{proposition}

\begin{proof}
By \eqref{eq:LMx-spectrum-list}--\eqref{eq:LMx-zero-branch-list}, the spectrum consists of eigenvalues
\[
    \abs{k_m}^2+\nu_n^2
    \quad(m\in\Z^2,n\ge1)
\]
with finite multiplicity, and eigenvalues \(\abs{k_m}^2\) with \(m\ne0\), also with finite multiplicity.  For every \(R>0\), only finitely many lattice pairs satisfy
\[
    \abs{k_m}^2+\nu_n^2\le R,
\]
and only finitely many \(m\in\Z^2\setminus\{0\}\) satisfy \(\abs{k_m}^2\le R\).  Hence the eigenvalues tend to infinity with finite multiplicity.  Therefore \((\LMx+I)^{-1}\) is compact.
\end{proof}

\begin{proposition}[Heat-trace admissibility]
\label{prop:heat-trace-admissibility}
For every \(\tau>0\) and every \(p\ge0\),
\begin{equation}
\label{eq:heat-trace-admissibility}
    \Tr\paren{\LMx^p e^{-\tau\LMx}}<\infty .
\end{equation}
In particular, the traces needed in the Gaussian construction are finite:
\begin{equation}
\label{eq:needed-heat-traces-finite}
    \Tr(e^{-\tau\LMx})<\infty,
    \qquad
    \Tr\paren{\LMx^{1/2}e^{-\tau\LMx}}<\infty,
    \qquad
    \Tr\paren{\LMx^{3/2}e^{-\tau\LMx}}<\infty .
\end{equation}
\end{proposition}

\begin{proof}
Using \eqref{eq:LMx-diagonal-trace-formula} with \(f(\lambda)=\lambda^p e^{-\tau\lambda}\),
\begin{align}
\label{eq:heat-trace-explicit}
    \Tr\paren{\LMx^p e^{-\tau\LMx}}
    &=
    2\sum_{m\in\Z^2}\sum_{n=1}^{\infty}
    \paren{\abs{k_m}^2+\nu_n^2}^p
    e^{-\tau(\abs{k_m}^2+\nu_n^2)}
    \\
    &\quad+
    \sum_{m\in\Z^2\setminus\{0\}}
    \abs{k_m}^{2p}e^{-\tau\abs{k_m}^2} .\nonumber
\end{align}
Let
\[
    c_{L,a}:=\min\set{\paren{\frac{2\pi}{L}}^2,\paren{\frac{\pi}{a}}^2}>0 .
\]
Then
\[
    \abs{k_m}^2+\nu_n^2
    \ge c_{L,a}\paren{\abs{m}^2+n^2}
    \qquad(n\ge1),
\]
and \(\abs{k_m}^2\ge (2\pi/L)^2\abs{m}^2\) for the normal branch.  A polynomial factor times a Gaussian lattice factor is summable on \(\Z^2\times\N\) and on \(\Z^2\).  Thus both sums in \eqref{eq:heat-trace-explicit} converge absolutely, proving \eqref{eq:heat-trace-admissibility}.  The three special cases in \eqref{eq:needed-heat-traces-finite} follow by taking \(p=0,1/2,3/2\).
\end{proof}

\begin{corollary}[Spectral hypotheses for the finite-volume Maxwell construction]
\label{cor:finite-volume-maxwell-spectral-hypotheses}
For each fixed finite \(L<\infty\) and \(a>0\), the reduced Maxwell operator \(\LMx\) is positive self-adjoint on \(\Hphys\), has compact resolvent, satisfies the gap estimate \eqref{eq:LMx-gap}, and obeys the heat-trace bounds \eqref{eq:heat-trace-admissibility}.  Thus it satisfies the finite-volume spectral hypotheses needed for the Riesz reduction and the heat-regularized Maxwell Gaussian source.
\end{corollary}

\begin{remark}[Use in the next section]
\label{rem:remaining-after-heat-trace}
This section supplies the finite-volume spectral input for the Riesz reduction
and Gaussian source construction of the next section.
\end{remark}
\section{Riesz reduction and the Maxwell Gaussian source}
\label{sec:riesz-gaussian-source}

Throughout this section the lateral size \(L<\infty\) and the plate spacing
\(a>0\) are fixed.  Thus \(\LMx\) denotes the reduced perfect-conductor Maxwell
operator on \(\Hphys\) constructed in Section~\ref{sec:maxwell-operator}.  By
Corollary~\ref{cor:finite-volume-maxwell-spectral-hypotheses}, \(\LMx\) is
positive self-adjoint, has compact resolvent, and has a positive finite-volume
spectral gap
\begin{equation}
    \LMx\geq \ell_{L,a}I,
    \qquad
    \ell_{L,a}
    :=
    \min\set{\paren{\frac{2\pi}{L}}^2,
              \paren{\frac{\pi}{a}}^2}>0 .
    \label{eq:section04-gap-assumption}
\end{equation}
It also satisfies the heat-trace admissibility statement
\begin{equation}
    \Tr\paren{\LMx^p e^{-\tau\LMx}}<\infty
    \qquad
    (\tau>0,\ p\geq0).
    \label{eq:section04-heat-trace-assumption}
\end{equation}
Consequently \(\LMx^{-1}\) is a bounded positive operator on \(\Hphys\) for
fixed finite \(L\).  The large-area limit is not taken in this section.

The purpose of this section is to prove the regulated representation theorem
for this finite-volume Maxwell operator.  The construction has two parts.  The
first is deterministic: the codimension-three transverse Riesz reduction of
\((\LMx+\abs{q}^2)^{-5/2}\) gives the reduced Maxwell Green operator
\(\LMx^{-1}\), up to the universal constant \((6\pi^2)^{-1}\).  The second is
stochastic: a prescribed heat-regularized Gaussian source with covariance
proportional to \(\LMx^{3/2}e^{-\tau\LMx}\) turns the quadratic form with
\(g\LMx^{-1}\) into the heat-regularized Maxwell trace.
This section is the point of closest formal overlap with
\cite{KhanKhan2026Scalar}.  At the level of spectral calculus, the
codimension-three Riesz reduction and the prescribed-covariance
quadratic-form identity are the same mechanism.  What changes here is the
underlying operator: instead of an abstract scalar operator, we apply the
construction to the reduced perfect-conductor Maxwell operator on \(\Hphys\).

\begin{remark}[Finite-volume nature of the inverse]
\label{rem:finite-volume-inverse-only}
The finite-volume lower bound in \eqref{eq:section04-gap-assumption} is used in
two places: to make the transverse Riesz integral an operator-norm integral and
to make \(\LMx^{-1}\) bounded in the stochastic quadratic form.  The lower bound
collapses as \(L\to\infty\), because of the \(n=0\) normal branch identified in
Section~\ref{sec:fourier-domain-heat-trace}.  Hence all identities below are
finite-volume identities.  Later, only after taking traces per unit area and
applying the parallel-plate interaction finite part, is the limit
\(L\to\infty\) taken.
\end{remark}

\subsection{An abstract codimension-three Riesz reduction}
\label{subsec:abstract-riesz-reduction}

We first record the Hilbert-space version of the transverse momentum integral.
It is independent of whether the operator acts on scalar functions, vector
fields, or sections of a vector bundle.  Only the spectral theorem and the
positive lower bound are used.

\begin{lemma}[Abstract transverse Riesz integral]
\label{lem:hilbert-space-riesz-integral}
Let \(A\) be a positive self-adjoint operator on a Hilbert space \(\mathcal H\),
and suppose
\[
    A\geq \ell I
    \qquad \text{for some }\ell>0 .
\]
For \(s>3/2\), define
\begin{equation}
    T_s(A)
    :=
    \int_{\R^3}\frac{\dd^3 q}{(2\pi)^3}
    (A+\abs{q}^2)^{-s} .
    \label{eq:TsA-def-section04}
\end{equation}
Then the integral in \eqref{eq:TsA-def-section04} converges in operator norm
and
\begin{equation}
    T_s(A)
    =
    \frac{1}{(4\pi)^{3/2}}
    \frac{\Gamma(s-3/2)}{\Gamma(s)}
    A^{3/2-s} .
    \label{eq:TsA-evaluated-section04}
\end{equation}
In particular,
\begin{equation}
    \int_{\R^3}\frac{\dd^3 q}{(2\pi)^3}
    (A+\abs{q}^2)^{-5/2}
    =
    \frac{1}{6\pi^2}A^{-1} .
    \label{eq:riesz-five-halves-abstract}
\end{equation}
\end{lemma}

\begin{proof}
The lower bound gives
\[
    \norm{(A+\abs{q}^2)^{-s}}
    \leq
    (\ell+\abs{q}^2)^{-s} .
\]
The right-hand side is integrable over \(\R^3\) exactly when \(s>3/2\).
Hence \eqref{eq:TsA-def-section04} is an operator-norm Bochner integral.

Let \(E_A(\dd\lambda)\) be the spectral resolution of \(A\).  For
\(u,v\in\mathcal H\), the preceding norm bound justifies Fubini and gives
\begin{align*}
    \inner{u}{T_s(A)v}
    &=
    \int_{\ell}^{\infty}
    \left[
        \int_{\R^3}\frac{\dd^3 q}{(2\pi)^3}
        (\lambda+\abs{q}^2)^{-s}
    \right]
    \dd\mu_{u,v}(\lambda),
\end{align*}
where \(\dd\mu_{u,v}(\lambda)=\inner{u}{E_A(\dd\lambda)v}\).  The scalar
integral is computed by the Schwinger representation:
\begin{align*}
    \int_{\R^3}\frac{\dd^3q}{(2\pi)^3}(\lambda+\abs{q}^2)^{-s}
    &=
    \frac{1}{\Gamma(s)}
    \int_0^\infty t^{s-1}e^{-t\lambda}
    \left[
        \int_{\R^3}\frac{\dd^3q}{(2\pi)^3}e^{-t\abs{q}^2}
    \right]\dd t                                                   \\
    &=
    \frac{1}{(4\pi)^{3/2}\Gamma(s)}
    \int_0^\infty t^{s-1-3/2}e^{-t\lambda}\dd t                  \\
    &=
    \frac{1}{(4\pi)^{3/2}}
    \frac{\Gamma(s-3/2)}{\Gamma(s)}
    \lambda^{3/2-s} .
\end{align*}
Substitution into the spectral integral proves
\eqref{eq:TsA-evaluated-section04}.  For \(s=5/2\),
\[
    \frac{1}{(4\pi)^{3/2}}\frac{\Gamma(1)}{\Gamma(5/2)}
    =
    \frac{1}{6\pi^2},
\]
which gives \eqref{eq:riesz-five-halves-abstract}.
\end{proof}

Applying Lemma~\ref{lem:hilbert-space-riesz-integral} with \(A=\LMx\) gives
the finite-volume Maxwell identity
\begin{equation}
    \int_{\R^3}\frac{\dd^3 q}{(2\pi)^3}
    (\LMx+\abs{q}^2)^{-5/2}
    =
    \frac{1}{6\pi^2}\LMx^{-1}.
    \label{eq:riesz-five-halves-LMx}
\end{equation}
This is the deterministic codimension-three reduction used below.

\subsection{The ambient product operator and transverse restriction}
\label{subsec:ambient-product-brane-restriction}

Let \(y\in\R^3\) denote the three auxiliary transverse coordinates.  The ambient
Hilbert space is
\begin{equation}
    \mathcal H_{\mathcal M}
    :=
    \Hphys\otimes L^2(\R^3_y),
    \label{eq:ambient-H-def-section04}
\end{equation}
and the ambient product operator is
\begin{equation}
    \Lamb
    :=
    \LMx\otimes I
    +
    I\otimes(-\Delta_y).
    \label{eq:ambient-product-operator-def}
\end{equation}
The Maxwell constraints and boundary conditions are already built into the
factor \(\Hphys\); the transverse factor is only the Euclidean auxiliary factor
used for the Riesz reduction.

\begin{proposition}[Ambient product operator]
\label{prop:ambient-product-operator-section04}
The operator \(\Lamb\) in \eqref{eq:ambient-product-operator-def} is positive
self-adjoint on \(\mathcal H_{\mathcal M}\).  Under Fourier transform in the
transverse variable \(y\), it is represented as the direct integral
\begin{equation}
    \widehat{\Lamb}
    =
    \int_{\R^3}^{\oplus}
    (\LMx+\abs{q}^2)\,
    \frac{\dd^3q}{(2\pi)^3} .
    \label{eq:direct-integral-Lamb}
\end{equation}
Consequently,
\begin{equation}
    \widehat{\Lamb^{-5/2}}
    =
    \int_{\R^3}^{\oplus}
    (\LMx+\abs{q}^2)^{-5/2}\,
    \frac{\dd^3q}{(2\pi)^3} .
    \label{eq:direct-integral-Lamb-minus-five-halves}
\end{equation}
\end{proposition}

\begin{proof}
By Proposition~\ref{prop:compact-resolvent-LMx}, \(\LMx\) has an orthonormal
eigenbasis \(\{u_j\}_{j\geq1}\) in \(\Hphys\), with
\[
    \LMx u_j=\lambda_j u_j,
    \qquad
    0<\ell_{L,a}\leq \lambda_1\leq\lambda_2\leq\cdots,
    \qquad
    \lambda_j\to\infty .
\]
Thus
\[
    \mathcal H_{\mathcal M}
    \cong
    \bigoplus_{j=1}^{\infty}L^2(\R^3_y),
    \qquad
    \Lamb
    \cong
    \bigoplus_{j=1}^{\infty}(\lambda_j-\Delta_y).
\]
Each operator \(\lambda_j-\Delta_y\) is positive self-adjoint.  The direct-sum
operator is therefore positive self-adjoint.  Fourier transform in \(y\) sends
\(-\Delta_y\) to multiplication by \(\abs{q}^2\), which gives the direct-integral
representation \eqref{eq:direct-integral-Lamb}.  Formula
\eqref{eq:direct-integral-Lamb-minus-five-halves} follows from the spectral
theorem.
\end{proof}

Let \(\kappa>0\) be the normalization of the ambient Riesz mediator and set
\begin{equation}
    g:=\frac{\kappa}{6\pi^2} .
    \label{eq:g-def-section04}
\end{equation}
Define the transversely reduced Maxwell Riesz mediator by the operator-norm
integral
\begin{equation}
    \VMx
    :=
    \kappa
    \int_{\R^3}\frac{\dd^3 q}{(2\pi)^3}
    (\LMx+\abs{q}^2)^{-5/2} .
    \label{eq:VMx-def-section04}
\end{equation}
By \eqref{eq:riesz-five-halves-LMx},
\begin{equation}
    \boxed{\VMx=g\LMx^{-1}.}
    \label{eq:VMx-equals-gLMxinv-section04}
\end{equation}
This is the Maxwell analogue of the scalar codimension-three Riesz reduction,
with \(\LMx\) replacing the corresponding scalar plate operator.

The same operator is obtained from \(\kappa\Lamb^{-5/2}\) by a canonical
mollified transverse restriction.  This avoids any illegal point evaluation in
the transverse \(L^2\)-factor.

\begin{proposition}[Mollified transverse restriction]
\label{prop:mollified-brane-restriction-Mx}
Let \(\eta\in C_c^\infty(\R^3)\) satisfy
\[
    \int_{\R^3}\eta(y)\dd y=1,
\]
and set
\[
    \eta_\eps(y):=\eps^{-3}\eta(y/\eps),
    \qquad \eps>0 .
\]
For \(J\in\Hphys\), define
\[
    R_\eps J:=J\otimes \eta_\eps
    \in\mathcal H_{\mathcal M} .
\]
Then for all \(J,K\in\Hphys\),
\begin{equation}
    \lim_{\eps\to0^+}
    \inner{R_\eps J}{\kappa\Lamb^{-5/2}R_\eps K}_{\mathcal H_{\mathcal M}}
    =
    \inner{J}{\VMx K}_{\Hphys}
    =
    \inner{J}{g\LMx^{-1}K}_{\Hphys} .
    \label{eq:mollified-brane-restriction-limit}
\end{equation}
The limit is independent of the normalized mollifier \(\eta\).
\end{proposition}

\begin{proof}
Use the Fourier transform convention in the transverse variable for which
Plancherel has measure \((2\pi)^{-3}\dd^3q\).  Then
\(\widehat{\eta_\eps}(q)=\widehat\eta(\eps q)\), and
Proposition~\ref{prop:ambient-product-operator-section04} gives
\begin{align*}
&\inner{R_\eps J}{\kappa\Lamb^{-5/2}R_\eps K}_{\mathcal H_{\mathcal M}} \\
&\qquad =
\kappa
\int_{\R^3}\frac{\dd^3q}{(2\pi)^3}
\abs{\widehat\eta(\eps q)}^2
\inner{J}{(\LMx+\abs{q}^2)^{-5/2}K}_{\Hphys} .
\end{align*}
Because \(\widehat\eta\) is bounded and \(\LMx\geq\ell_{L,a}I\),
\[
    \abs{\inner{J}{(\LMx+\abs{q}^2)^{-5/2}K}}
    \leq
    \norm{J}\norm{K}(\ell_{L,a}+\abs{q}^2)^{-5/2},
\]
and the majorant is integrable over \(\R^3\).  Also \(\widehat\eta(0)=1\).
Dominated convergence therefore gives
\begin{align*}
    \lim_{\eps\to0^+}
    \inner{R_\eps J}{\kappa\Lamb^{-5/2}R_\eps K}
    &=
    \kappa
    \int_{\R^3}\frac{\dd^3q}{(2\pi)^3}
    \inner{J}{(\LMx+\abs{q}^2)^{-5/2}K} \\
    &=
    \inner{J}{\VMx K}.
\end{align*}
Equation~\eqref{eq:VMx-equals-gLMxinv-section04} gives the final equality.
Since only \(\widehat\eta(0)=1\) enters the limiting value, the limit is
independent of the particular mollifier.
\end{proof}

\begin{remark}[Meaning of the ambient mediator]
\label{rem:fractional-ambient-mediator-Mx}
The operator \(\Lamb^{-5/2}\) is a fractional Riesz-type mediator on the product
space \(\Hphys\otimes L^2(\R^3_y)\).  It is not the ordinary Green operator
\(\Lamb^{-1}\) of a local higher-dimensional Maxwell theory.  Its role here is
narrower: after transverse restriction, the exponent \(5/2\) produces the
ordinary reduced Maxwell Green operator \(\LMx^{-1}\) on the physical plate
Hilbert space.
\end{remark}

\subsection{The prescribed Maxwell Gaussian source}
\label{subsec:maxwell-gaussian-source}

Let \(\{u_j\}_{j\geq1}\) be an orthonormal eigenbasis of \(\LMx\) in
\(\Hphys\),
\[
    \LMx u_j=\lambda_j u_j,
    \qquad \lambda_j>0 .
\]
Let \(\XiNoise\) denote cylindrical white noise on \(\Hphys\), formally
\begin{equation}
    \XiNoise=\sum_{j\geq1}\xi_j u_j .
    \label{eq:cylindrical-white-noise-Mx}
\end{equation}
For the complex Hilbert-space convention, the coefficients may be taken as
centered circular complex Gaussians satisfying
\[
    \E[\xi_i\overline{\xi_j}]=\delta_{ij},
    \qquad
    \E[\xi_i\xi_j]=0 .
\]
The real convention gives the same covariance formulas with complex conjugates
removed.  The formal series \eqref{eq:cylindrical-white-noise-Mx} is not, in
general, an \(\Hphys\)-valued random variable; see
\cite{Janson1997} for the Hilbert-space Gaussian background.

For \(\tau>0\), define
\begin{equation}
    \boxed{
    \Sigmatau
    :=
    \paren{\frac{\hbarc}{g}}^{1/2}
    \LMx^{3/4}e^{-\tau\LMx/2}\XiNoise .
    }
    \label{eq:Sigma-tau-def-section04}
\end{equation}
Equivalently,
\begin{equation}
    \Sigmatau
    =
    \paren{\frac{\hbarc}{g}}^{1/2}
    \sum_{j\geq1}
    \lambda_j^{3/4}e^{-\tau\lambda_j/2}\xi_j u_j .
    \label{eq:Sigma-tau-components}
\end{equation}

\begin{proposition}[Heat-regularized Maxwell source]
\label{prop:maxwell-source-hilbert-valued}
For every \(\tau>0\), the series \eqref{eq:Sigma-tau-components} converges in
\(L^2(\Omega_{\mathrm{prob}};\Hphys)\).  Hence \(\Sigmatau\) is an
\(\Hphys\)-valued centered Gaussian random variable.  Its covariance operator is
\begin{equation}
    \Ctau
    :=
    \E[\Sigmatau\otimes\Sigmatau^*]
    =
    \frac{\hbarc}{g}\LMx^{3/2}e^{-\tau\LMx} .
    \label{eq:Ctau-def-section04}
\end{equation}
In particular \(\Ctau\) is positive and trace class.
\end{proposition}

\begin{proof}
By orthonormality and the Gaussian normalization,
\begin{align*}
    \E\norm{\Sigmatau}_{\Hphys}^2
    &=
    \frac{\hbarc}{g}
    \sum_{j\geq1}\lambda_j^{3/2}e^{-\tau\lambda_j}          \\
    &=
    \frac{\hbarc}{g}
    \Tr\paren{\LMx^{3/2}e^{-\tau\LMx}} .
\end{align*}
The trace is finite by \eqref{eq:section04-heat-trace-assumption}, with
\(p=3/2\).  This proves convergence in
\(L^2(\Omega_{\mathrm{prob}};\Hphys)\) and therefore defines an
\(\Hphys\)-valued centered Gaussian random variable.

The covariance operator is characterized by its matrix entries in the
orthonormal basis \(\{u_j\}\).  With the complex convention stated above,
\[
    \E\!\brac{
        \inner{u_i}{\Sigmatau}
        \overline{\inner{u_j}{\Sigmatau}}
    }
    =
    \frac{\hbarc}{g}
    \lambda_j^{3/2}e^{-\tau\lambda_j}\delta_{ij},
\]
with the analogous formula in the real convention.  Hence \(\Ctau\) is diagonal
in this basis with eigenvalues
\((\hbarc/g)\lambda_j^{3/2}e^{-\tau\lambda_j}\), which is exactly the spectral
operator \((\hbarc/g)\LMx^{3/2}e^{-\tau\LMx}\).  Its trace is the finite number
computed above.
\end{proof}

\begin{remark}[Prescribed covariance]
\label{rem:maxwell-covariance-prescribed}
The covariance \eqref{eq:Ctau-def-section04} is prescribed.  It is the Maxwell
analogue of the scalar covariance chosen so that pairing with the induced Green
operator converts the multiplier \(\LMx^{3/2}\) into \(\LMx^{1/2}\).  Thus the
present section proves the resulting representation theorem for this prescribed
covariance.
\end{remark}

\begin{remark}[Projection is spectral, not a flat-space Fourier projector]
\label{rem:no-flat-projector}
The projection used here is the spectral projection encoded by \(\Hphys\) and
\(\LMx\) in the bounded plate geometry.  No separate flat-space tensor
projector \(\delta_{ij}-q_iq_j/\abs{q}^2\) enters the construction.  The
physical transversality and boundary conditions are encoded by the Hilbert
space \(\Hphys\) and the operator \(\LMx\).  The
covariance \eqref{eq:Ctau-def-section04} is therefore the spectral covariance
on the already reduced Maxwell Hilbert space.
\end{remark}

\subsection{The regulated quadratic-form trace identity}
\label{subsec:regulated-maxwell-trace-identity}

Define the regulated Maxwell quadratic energy by
\begin{equation}
    \boxed{
    \UtauMx
    :=
    \frac12
    \inner{\Sigmatau}{g\LMx^{-1}\Sigmatau}_{\Hphys} .
    }
    \label{eq:Utau-Mx-def-section04}
\end{equation}
For finite \(L\), the operator \(g\LMx^{-1}\) is bounded and positive.  Hence
\(\UtauMx\) is well-defined almost surely and is nonnegative.  Here the inner
product is the Hermitian inner product on \(\Hphys\), so the positive operator
\(g\LMx^{-1}\) defines a nonnegative quadratic form.  Moreover,
\[
    \E\abs{\UtauMx}
    \leq
    \frac{g}{2}\norm{\LMx^{-1}}\,
    \E\norm{\Sigmatau}_{\Hphys}^2
    <\infty .
\]

\begin{theorem}[Regulated Maxwell quadratic-form representation]
\label{thm:regulated-maxwell-quadratic-trace}
For every \(\tau>0\),
\begin{equation}
    \boxed{
    \E[\UtauMx]
    =
    \frac{\hbarc}{2}
    \Trphys\paren{\LMx^{1/2}e^{-\tau\LMx}} .
    }
    \label{eq:regulated-maxwell-trace-identity}
\end{equation}
The trace on the right is finite.
\end{theorem}

\begin{proof}
Using the eigen-expansion \eqref{eq:Sigma-tau-components},
\[
    \UtauMx
    =
    \frac{g}{2}
    \sum_{j\geq1}\lambda_j^{-1}
    \abs{\inner{u_j}{\Sigmatau}}^2 .
\]
For each \(j\),
\[
    \E\abs{\inner{u_j}{\Sigmatau}}^2
    =
    \frac{\hbarc}{g}\lambda_j^{3/2}e^{-\tau\lambda_j} .
\]
Since the terms are nonnegative, monotone convergence gives
\begin{align*}
    \E[\UtauMx]
    &=
    \frac{g}{2}
    \sum_{j\geq1}\lambda_j^{-1}
    \frac{\hbarc}{g}\lambda_j^{3/2}e^{-\tau\lambda_j} \\
    &=
    \frac{\hbarc}{2}
    \sum_{j\geq1}\lambda_j^{1/2}e^{-\tau\lambda_j}       \\
    &=
    \frac{\hbarc}{2}
    \Trphys\paren{\LMx^{1/2}e^{-\tau\LMx}} .
\end{align*}
The final trace is finite by \eqref{eq:section04-heat-trace-assumption}, with
\(p=1/2\).
\end{proof}

Equivalently, the same calculation can be written as the Hilbert-space
covariance trace identity
\begin{equation}
    \E[\UtauMx]
    =
    \frac12\Tr\paren{g\LMx^{-1}\Ctau}
    =
    \frac{\hbarc}{2}\Trphys\paren{\LMx^{1/2}e^{-\tau\LMx}} .
    \label{eq:covariance-trace-identity-Mx}
\end{equation}
Here \(g\LMx^{-1}\) is bounded in finite volume and \(\Ctau\) is trace class.

\begin{remark}[Regulated identity, not yet the plate finite part]
\label{rem:regulated-not-yet-finite-part}
Theorem~\ref{thm:regulated-maxwell-quadratic-trace} is an identity for every
fixed \(\tau>0\) and finite \(L\).  The trace evaluation requires the Maxwell
TE/TM spectral decomposition, and the electromagnetic plate coefficient
requires the standard parallel-plate interaction finite-part prescription.
Those steps are carried out in the next sections.
\end{remark}

\begin{remark}[Finite parts of positive regulated quantities]
\label{rem:positive-regulated-negative-fp-Mx}
For each fixed \(\tau>0\), \(\UtauMx\geq0\) almost surely.  A later
renormalized interaction finite part can nevertheless be negative, because it
is obtained after subtracting divergent local and separation-independent terms
from the regulated expectation.  This is the usual situation in Casimir
finite-part calculations and is not a contradiction.
\end{remark}

\subsection{Summary of the finite-volume representation}
\label{subsec:finite-volume-representation-summary}

Combining \eqref{eq:VMx-equals-gLMxinv-section04},
\eqref{eq:Sigma-tau-def-section04}, and
Theorem~\ref{thm:regulated-maxwell-quadratic-trace}, we have proved the
finite-volume representation
\begin{equation}
    \VMx=g\LMx^{-1},
    \qquad
    \Sigmatau
    =
    \paren{\frac{\hbarc}{g}}^{1/2}
    \LMx^{3/4}e^{-\tau\LMx/2}\XiNoise,
    \label{eq:combined-representation-04}
\end{equation}
with
\begin{equation}
    \Ctau
    =
    \E[\Sigmatau\otimes\Sigmatau^*]
    =
    \frac{\hbarc}{g}\LMx^{3/2}e^{-\tau\LMx},
    \label{eq:combined-covariance-04}
\end{equation}
and
\begin{equation}
    \E\left[
        \frac12\inner{\Sigmatau}{\VMx\Sigmatau}_{\Hphys}
    \right]
    =
    \frac{\hbarc}{2}
    \Trphys\paren{\LMx^{1/2}e^{-\tau\LMx}} .
    \label{eq:combined-trace-identity-04}
\end{equation}
This is the Maxwell analogue of the scalar Riesz/Gaussian trace representation
at fixed finite lateral volume.

\begin{remark}[Use in the next sections]
\label{rem:remaining-after-riesz-gaussian}
The regulated identity \eqref{eq:combined-trace-identity-04} is the starting
point for the next two sections, which identify the physical Maxwell trace and
evaluate its standard parallel-plate interaction finite part.
\end{remark}
\section{The TE/TM trace decomposition}
\label{sec:te-tm-decomposition}

The previous section proved the regulated Maxwell quadratic-form identity
\begin{equation}
    \E[\UtauMx]
    =
    \frac{\hbarc}{2}
    \Trphys\paren{\LMx^{1/2}e^{-\tau\LMx}}
    \qquad (\tau>0)
    \label{eq:te-tm-starting-regulated-trace}
\end{equation}
for the finite slab \(\OmegaLa=\TtwoL\times[0,a]\).  It remains to identify
this physical Maxwell trace with scalar plate traces.  This section proves the
finite-volume spectral equivalence
\begin{equation}
    \LMx \simeq \LD\oplus \LNp,
    \label{eq:LMx-spectral-equivalence-LD-LNp-informal}
\end{equation}
where \(\LD\) is the scalar Dirichlet Laplacian in the plate direction and
\(\LNp\) is the scalar Neumann Laplacian with its constant zero mode removed.
The prime is essential: it corresponds exactly to removing the static normal
Maxwell mode \(\hzero\).
This is the step with no scalar analogue in \cite{KhanKhan2026Scalar}.  The
scalar paper reduces the plate benchmark to scalar traces; the present section
must prove that the physical Maxwell trace is exactly \(\LD\oplus\LNp\), with
the electromagnetic multiplicity coming from Maxwell mode structure rather than
from a hand-inserted factor of two.

The argument is finite-volume and spectral.  We do not use a vector potential
or a gauge-fixed determinant.  The physical degrees of freedom have already
been built into the electric-field Hilbert space \(\Hphys\), and the Fourier
coefficient lemma in Section~\ref{sec:fourier-domain-heat-trace} identifies
that form domain with the mixed sine/cosine coefficient space.  The present
section reorganizes that coefficient-space diagonalization into the standard
TE/TM language and then into scalar Dirichlet and reduced Neumann traces.
The phrase TE/TM is used for the nonzero lateral-momentum blocks; at zero
lateral momentum the labels are noncanonical, and the invariant statement is
the spectral equivalence \(\LMx\simeq \LD\oplus\LNp\).

\subsection{Scalar comparison operators}
\label{subsec:scalar-comparison-operators}

Let \(\LD\) denote the positive scalar Laplacian on \(\OmegaLa\) with periodic
boundary conditions in the lateral variables and Dirichlet boundary conditions
at \(z=0,a\).  Thus the separated orthonormal eigenfunctions are
\begin{equation}
    d_{m,n}(r,z):=\phim(r)\sn(z),
    \qquad
    m\in\Z^2,
    \quad n\geq1,
    \label{eq:dirichlet-scalar-eigenfunctions}
\end{equation}
with eigenvalues
\begin{equation}
    \lambdamn
    =
    \abs{k_m}^2+\nu_n^2,
    \qquad
    k_m=\frac{2\pi}{L}m,
    \quad
    \nu_n=\frac{\pi n}{a} .
    \label{eq:scalar-lambda-mn-section05}
\end{equation}
Let \(\LN\) denote the positive scalar Laplacian with periodic lateral
conditions and Neumann boundary conditions at \(z=0,a\).  Its separated
orthonormal eigenfunctions are
\begin{equation}
    u_{m,n}(r,z):=\phim(r)\cn(z),
    \qquad
    m\in\Z^2,
    \quad n\geq0,
    \label{eq:neumann-scalar-eigenfunctions}
\end{equation}
with the same eigenvalues \(\lambdamn=\abs{k_m}^2+\nu_n^2\), now allowing
\(n=0\).  The mode \(u_{0,0}\) is the scalar constant mode and has eigenvalue
zero.  We write \(\LNp\) for the restriction of \(\LN\) to
\(u_{0,0}^{\perp}\).  Equivalently,
\begin{equation}
    \Tr' f(\LN)
    :=
    \Tr f(\LNp)
    =
    \sum_{\substack{m\in\Z^2,\ n\geq0\\ (m,n)\ne(0,0)}}
    f\paren{\abs{k_m}^2+\nu_n^2},
    \label{eq:prime-neumann-trace-def-section05}
\end{equation}
whenever the sum is absolutely convergent.

Thus
\begin{align}
    \Tr f(\LD)
    &=
    \sum_{m\in\Z^2}\sum_{n=1}^{\infty}
    f\paren{\abs{k_m}^2+\nu_n^2},
    \label{eq:dirichlet-trace-section05}\\
    \Tr' f(\LN)
    &=
    \sum_{m\in\Z^2}\sum_{n=1}^{\infty}
    f\paren{\abs{k_m}^2+\nu_n^2}
    +
    \sum_{m\in\Z^2\setminus\{0\}}
    f\paren{\abs{k_m}^2}.
    \label{eq:neumann-prime-trace-section05}
\end{align}
Adding the two traces gives
\begin{equation}
    \Tr f(\LD)+\Tr' f(\LN)
    =
    2\sum_{m\in\Z^2}\sum_{n=1}^{\infty}
    f\paren{\abs{k_m}^2+\nu_n^2}
    +
    \sum_{m\in\Z^2\setminus\{0\}}
    f\paren{\abs{k_m}^2} .
    \label{eq:LD-plus-LNp-trace-expanded}
\end{equation}
This is the scalar expression to be matched with the Maxwell trace.

\begin{remark}[Complex modes and real fields]
\label{rem:complex-mode-convention-section05}
The paper works over the complex Hilbert space.  Thus the Fourier modes
\(m\) and \(-m\) are distinct basis vectors, as usual in complex spectral
calculus.  A real-field formulation would impose the corresponding conjugation
condition on coefficients, but it gives the same spectral multiplicities and
trace formulas.
\end{remark}

\subsection{Physical Maxwell blocks}
\label{subsec:physical-maxwell-blocks-section05}

We recall the coefficient blocks from Section~\ref{sec:fourier-domain-heat-trace}.  For
\(n\geq1\), the Maxwell coefficients at fixed \((m,n)\) are
\((a_{m,n},b_{m,n})\in\C^2\oplus\C\) subject to the divergence constraint
\begin{equation}
    \ii k_m\cdot a_{m,n}-\nu_n b_{m,n}=0 .
    \label{eq:divergence-block-section05}
\end{equation}
The corresponding block is
\begin{equation}
    W_{m,n}
    =
    \set{(a,b)\in\C^2\oplus\C:
          \ii k_m\cdot a-\nu_n b=0}.
    \label{eq:Wmn-section05}
\end{equation}
Since \(\nu_n>0\), this is a two-dimensional subspace of \(\C^3\).  For
\(n=0\), the only Maxwell coefficient is the normal coefficient \(b_{m,0}\).
The branch \((m,n)=(0,0)\) is absent from \(\Hphys\), because it is the removed
static mode \(\hzero\).  Thus for \(m\ne0\) the \(n=0\) block is one-dimensional.

By Proposition~\ref{prop:diagonal-trace-formula}, \(\LMx\) acts as scalar
multiplication by \(\lambdamn\) on each two-dimensional block \(W_{m,n}\),
\(n\geq1\), and by \(\abs{k_m}^2\) on each one-dimensional \(n=0\) block,
\(m\ne0\).  The present section chooses explicit orthonormal bases for these
blocks.  These bases are the finite-volume TE and TM modes.

\subsection{TE and TM modes for nonzero lateral momentum}
\label{subsec:te-tm-nonzero-lateral-momentum}

Assume first that \(m\ne0\), so that \(k_m\ne0\).  Define the lateral unit
vectors
\begin{equation}
    \khat:=\frac{k_m}{\abs{k_m}},
    \qquad
    \khatperp:=\frac{1}{\abs{k_m}}(-k_{m,2},k_{m,1}) .
    \label{eq:khat-khatperp-def-section05}
\end{equation}
Here \(k_m=(k_{m,1},k_{m,2})\).  For \(n\geq1\), define the TE mode
\begin{equation}
    E^{\mathrm{TE}}_{m,n}(r,z)
    :=
    \phim(r)\sn(z)\,\khatperp .
    \label{eq:TE-mode-def-section05}
\end{equation}
It has coefficient vector
\begin{equation}
    a_{m,n}=\khatperp,
    \qquad
    b_{m,n}=0 .
    \label{eq:TE-coeff-section05}
\end{equation}
The divergence constraint holds because \(k_m\cdot\khatperp=0\).  The PEC
condition holds because the tangential components are proportional to \(\sn\),
which vanishes at the plates.

Define the TM mode by
\begin{equation}
    E^{\mathrm{TM}}_{m,n}(r,z)
    :=
    \phim(r)
    \left[
        -\ii\frac{\nu_n}{\sqrt{\lambdamn}}\,\khat\,\sn(z)
        +
        \frac{\abs{k_m}}{\sqrt{\lambdamn}}\, e_z\,\cn(z)
    \right] .
    \label{eq:TM-mode-def-section05}
\end{equation}
Its coefficient vector is
\begin{equation}
    a_{m,n}
    =
    -\ii\frac{\nu_n}{\sqrt{\lambdamn}}\,\khat,
    \qquad
    b_{m,n}
    =
    \frac{\abs{k_m}}{\sqrt{\lambdamn}} .
    \label{eq:TM-coeff-section05}
\end{equation}
The divergence constraint is
\begin{align*}
    \ii k_m\cdot a_{m,n}-\nu_n b_{m,n}
    &=
    \ii\abs{k_m}
    \left(-\ii\frac{\nu_n}{\sqrt{\lambdamn}}\right)
    -
    \nu_n\frac{\abs{k_m}}{\sqrt{\lambdamn}} \\
    &=0 .
\end{align*}
Again the tangential components are proportional to \(\sn\), so the PEC
condition is satisfied.

The two modes \eqref{eq:TE-mode-def-section05} and
\eqref{eq:TM-mode-def-section05} are orthonormal.  The TE mode has norm one.
The TM mode has norm
\begin{equation}
    \frac{\nu_n^2}{\lambdamn}
    +
    \frac{\abs{k_m}^2}{\lambdamn}
    =1,
    \label{eq:TM-normalization-section05}
\end{equation}
using \(\lambdamn=\abs{k_m}^2+\nu_n^2\).  Their inner product is zero because
\(\khatperp\cdot\khat=0\) and the TE mode has no normal component.  Hence
\(E^{\mathrm{TE}}_{m,n}\) and \(E^{\mathrm{TM}}_{m,n}\) form an orthonormal basis
of the two-dimensional block \(W_{m,n}\) for \(m\ne0\), \(n\geq1\).

Since \(\LMx\) is scalar multiplication by \(\lambdamn\) on the entire block
\(W_{m,n}\), both modes are eigenvectors:
\begin{equation}
    \LMx E^{\mathrm{TE}}_{m,n}
    =
    \lambdamn E^{\mathrm{TE}}_{m,n},
    \qquad
    \LMx E^{\mathrm{TM}}_{m,n}
    =
    \lambdamn E^{\mathrm{TM}}_{m,n} .
    \label{eq:TE-TM-eigenvalues-section05}
\end{equation}

\subsection{The zero lateral-momentum blocks}
\label{subsec:zero-lateral-momentum-blocks-section05}

When \(m=0\) and \(n\geq1\), there is no distinguished lateral direction
\(\khat\).  The divergence constraint reduces to
\begin{equation}
    -\nu_n b_{0,n}=0,
\end{equation}
so \(b_{0,n}=0\), while the two tangential coefficients are free.  A convenient
orthonormal basis is
\begin{equation}
    E^{x}_{0,n}(r,z):=\phi_0(r)\sn(z)e_x,
    \qquad
    E^{y}_{0,n}(r,z):=\phi_0(r)\sn(z)e_y .
    \label{eq:m0-transverse-modes-section05}
\end{equation}
Both modes are divergence-free, satisfy the PEC condition, and have eigenvalue
\begin{equation}
    \lambda_{0,n}=\nu_n^2 .
    \label{eq:m0-eigenvalue-section05}
\end{equation}
The TE/TM labels are not canonical at zero lateral momentum, but the physical
multiplicity is still exactly two.

\subsection{The normal zero-vertical branch}
\label{subsec:normal-zero-vertical-branch-section05}

For \(n=0\), there is no tangential sine mode.  For each \(m\ne0\), define
\begin{equation}
    E^{0}_{m}(r,z):=\phim(r)c_0(z)e_z .
    \label{eq:normal-zero-branch-mode-section05}
\end{equation}
This field is divergence-free because \(\partial_z c_0=0\), and it satisfies
\(n\times E=0\) because it is purely normal to the plates.  Its eigenvalue is
\begin{equation}
    \LMx E^{0}_{m}=\abs{k_m}^2 E^{0}_{m} .
    \label{eq:normal-zero-branch-eigenvalue-section05}
\end{equation}
The missing \(m=0\) member of this branch is \(c_0 e_z\), equivalently
\(\hzero\), the removed static normal mode.  Thus the branch contributes one
physical mode for each \(m\in\Z^2\setminus\{0\}\).

\subsection{Completeness and unitary equivalence}
\label{subsec:completeness-unitary-equivalence-section05}

The preceding modes form a complete orthonormal basis of \(\Hphys\).  Indeed,
for \(m\ne0\), \(n\geq1\), the pair
\(E^{\mathrm{TE}}_{m,n},E^{\mathrm{TM}}_{m,n}\) is an orthonormal basis of the
block \(W_{m,n}\).  For \(m=0\), \(n\geq1\), the pair
\(E^{x}_{0,n},E^{y}_{0,n}\) is an orthonormal basis of the corresponding
block.  For \(n=0\), the vectors \(E^0_m\), \(m\ne0\), span the one-dimensional
normal blocks.  The Fourier coefficient characterization of \(\Hphys\) is the
orthogonal Hilbert direct sum of precisely these blocks, with the static
\((m,n)=(0,0)\) normal block removed.

This gives an explicit unitary equivalence with \(\LD\oplus\LNp\).  Let
\(\mathcal H_D:=L^2(\OmegaLa)\) be the scalar Dirichlet Hilbert space and let
\(\mathcal H_N':=u_{0,0}^{\perp}\subset L^2(\OmegaLa)\) be the reduced Neumann
Hilbert space.  Define a unitary map
\begin{equation}
    \mathcal U:
    \mathcal H_D\oplus\mathcal H_N'
    \longrightarrow
    \Hphys
    \label{eq:unitary-U-def-section05}
\end{equation}
first on the separated scalar eigenbasis by the following rules.

For \(m\ne0\), \(n\geq1\), set
\begin{equation}
    \mathcal U(d_{m,n},0)=E^{\mathrm{TE}}_{m,n},
    \qquad
    \mathcal U(0,u_{m,n})=E^{\mathrm{TM}}_{m,n} .
    \label{eq:unitary-map-nonzero-m-section05}
\end{equation}
For \(m=0\), \(n\geq1\), set
\begin{equation}
    \mathcal U(d_{0,n},0)=E^{x}_{0,n},
    \qquad
    \mathcal U(0,u_{0,n})=E^{y}_{0,n} .
    \label{eq:unitary-map-zero-m-section05}
\end{equation}
For the reduced Neumann \(n=0\) branch, set
\begin{equation}
    \mathcal U(0,u_{m,0})=E^0_m,
    \qquad
    m\in\Z^2\setminus\{0\} .
    \label{eq:unitary-map-normal-branch-section05}
\end{equation}
Because both sides are mapped between complete orthonormal eigenbases, this
extends uniquely to a unitary operator.  Moreover, it intertwines the scalar
direct-sum operator and the Maxwell operator:
\begin{equation}
    \LMx\mathcal U
    =
    \mathcal U(\LD\oplus\LNp)
    \quad
    \text{on the operator domain of }\LD\oplus\LNp .
    \label{eq:intertwining-LMx-LD-LNp-section05}
\end{equation}
Equivalently,
\begin{equation}
    \boxed{
    \LMx
    \simeq
    \LD\oplus\LNp .
    }
    \label{eq:LMx-unitarily-equivalent-LD-LNp}
\end{equation}

\begin{remark}[What is and is not canonical]
\label{rem:noncanonical-zero-m-split-section05}
For \(m\ne0\), the labels TE and TM are tied to the lateral direction \(k_m\),
and the split in \eqref{eq:unitary-map-nonzero-m-section05} is canonical up to
phase conventions.  For \(m=0\), there is no lateral direction, so the split
\eqref{eq:unitary-map-zero-m-section05} is only a convenient choice of basis in
a two-dimensional degenerate eigenspace.  The trace identity below is
independent of that choice.
\end{remark}

\subsection{Trace identity}
\label{subsec:te-tm-trace-identity-section05}

The unitary equivalence immediately gives the trace identity.  We state it in
a form adapted to the later regulated and finite-part calculations.

\begin{theorem}[Maxwell trace as Dirichlet plus reduced Neumann]
\label{thm:maxwell-trace-Dirichlet-Neumann}
Let \(f:[0,\infty)\to\C\) be Borel.  Suppose that the scalar sums in
\eqref{eq:dirichlet-trace-section05} and \eqref{eq:neumann-prime-trace-section05}
converge absolutely.  Then
\begin{equation}
    \boxed{
    \Trphys f(\LMx)
    =
    \Tr f(\LD)+\Tr' f(\LN)
    =
    \Tr f(\LD)+\Tr f(\LNp).
    }
    \label{eq:maxwell-trace-Dirichlet-plus-Neumann}
\end{equation}
Equivalently,
\begin{equation}
    \Trphys f(\LMx)
    =
    2\sum_{m\in\Z^2}\sum_{n=1}^{\infty}
    f\paren{\abs{k_m}^2+\nu_n^2}
    +
    \sum_{m\in\Z^2\setminus\{0\}}
    f\paren{\abs{k_m}^2} .
    \label{eq:maxwell-trace-expanded-section05}
\end{equation}
\end{theorem}

\begin{proof}
By \eqref{eq:LMx-unitarily-equivalent-LD-LNp}, functional calculus gives
\begin{equation*}
    f(\LMx)
    =
    \mathcal U\bigl(f(\LD)\oplus f(\LNp)\bigr)\mathcal U^{-1}
\end{equation*}
whenever these spectral operators are trace class.  Taking traces gives
\eqref{eq:maxwell-trace-Dirichlet-plus-Neumann}.  Expanding the scalar traces
using \eqref{eq:dirichlet-trace-section05} and
\eqref{eq:neumann-prime-trace-section05} gives
\eqref{eq:maxwell-trace-expanded-section05}.  The absolute-convergence
hypothesis justifies taking the traces as sums over the displayed eigenbases.
\end{proof}

For the heat-regularized functions used in this paper, the hypotheses of
Theorem~\ref{thm:maxwell-trace-Dirichlet-Neumann} are satisfied by the
heat-trace admissibility proved in Proposition~\ref{prop:heat-trace-admissibility}.  In particular, for every \(\tau>0\),
\begin{equation}
    f_\tau(\lambda):=\lambda^{1/2}e^{-\tau\lambda}
    \label{eq:f-tau-sqrt-heat-section05}
\end{equation}
is trace-admissible, and
\begin{equation}
    \Trphys\paren{\LMx^{1/2}e^{-\tau\LMx}}
    =
    \Tr\paren{\LD^{1/2}e^{-\tau\LD}}
    +
    \Tr'\paren{\LN^{1/2}e^{-\tau\LN}} .
    \label{eq:regulated-maxwell-trace-LD-LN-section05}
\end{equation}
Combining this with Theorem~\ref{thm:regulated-maxwell-quadratic-trace} gives
\begin{equation}
    \boxed{
    \E[\UtauMx]
    =
    \frac{\hbarc}{2}
    \left[
        \Tr\paren{\LD^{1/2}e^{-\tau\LD}}
        +
        \Tr'\paren{\LN^{1/2}e^{-\tau\LN}}
    \right].
    }
    \label{eq:regulated-stochastic-trace-LD-LN-section05}
\end{equation}

\subsection{Explicit finite-volume sum}
\label{subsec:explicit-finite-volume-sum-section05}

Using \eqref{eq:maxwell-trace-expanded-section05} with
\(f=f_\tau\), the regulated Maxwell trace can be written as
\begin{equation}
    \E[\UtauMx]
    =
    \frac{\hbarc}{2}
    \left[
        2\sum_{m\in\Z^2}\sum_{n=1}^{\infty}
        \lambdamn^{1/2}e^{-\tau\lambdamn}
        +
        \sum_{m\in\Z^2\setminus\{0\}}
        \abs{k_m}e^{-\tau\abs{k_m}^2}
    \right].
    \label{eq:regulated-maxwell-trace-explicit-section05}
\end{equation}
The first double sum is the combined contribution of the two physical
polarizations for every nonzero vertical mode.  The final sum is the reduced
Neumann \(n=0\) branch.  It is a genuine finite-volume Maxwell branch, not an
error or an extra scalar channel.  Its dependence on \(a\) is absent, and that
fact is used only later, when the standard parallel-plate interaction finite
part is taken.

\begin{remark}[No scalar doubling shortcut]
\label{rem:no-scalar-doubling-shortcut-section05}
The factor of two in the first term of
\eqref{eq:regulated-maxwell-trace-explicit-section05} has not been inserted by
hand.  It comes from the two-dimensional divergence-free Maxwell block
\(W_{m,n}\) for each \(n\geq1\), together with the PEC boundary condition.  The
additional \(n=0\) branch and the removed static mode are both visible in the
operator spectrum.  Thus \eqref{eq:maxwell-trace-Dirichlet-plus-Neumann} is a
Maxwell trace identity, not a scalar two-channel assumption.
\end{remark}

\begin{remark}[Use in the next section]
\label{rem:remaining-after-te-tm-section05}
This decomposition rewrites the regulated stochastic Maxwell expectation as
scalar Dirichlet and reduced Neumann traces.  The next section evaluates their
standard parallel-plate interaction finite parts.
\end{remark}
\section{Finite part and the electromagnetic plate coefficient}
\label{sec:finite-part-plate-result}

The preceding sections proved the finite-volume representation
\begin{equation}
\label{eq:section06-starting-point}
    \E[\UtauMx]
    =
    \frac{\hbarc}{2}\Trphys\paren{\LMx^{1/2}e^{-\tau\LMx}}
\end{equation}
and the finite-volume spectral identity
\begin{equation}
\label{eq:section06-maxwell-DN-trace-start}
    \Trphys f(\LMx)=\Tr f(\LD)+\Tr' f(\LN).
\end{equation}
The one-channel scalar Dirichlet finite part used below is the same standard
parallel-plate scalar input that appears in \cite{KhanKhan2026Scalar}.  The
new point is how that scalar input enters the Maxwell problem: through the
exact decomposition \(\Trphys f(\LMx)=\Tr f(\LD)+\Tr' f(\LN)\), together with
the \(a\)-independent reduced Neumann branch.
This section evaluates the standard parallel-plate interaction finite part of
that trace.  The order of operations is fixed throughout:
\begin{equation}
\label{eq:section06-order-of-operations}
\begin{aligned}
    L<\infty
    &\quad\longrightarrow\quad
    \text{regulated identity}
    \quad\longrightarrow\quad
    \text{energy density at fixed }\tau \\[2pt]
    &\quad\longrightarrow\quad
    L\to\infty
    \quad\longrightarrow\quad
    \FP_{\mathrm{int},\,\tau\to0^+}.
\end{aligned}
\end{equation}
We do not take a large-area limit of the bounded inverse \(\LMx^{-1}\).  The
large-area limit is taken only after the heat-regularized trace density has
been formed.

\subsection{The regulated finite-box Maxwell trace}
\label{subsec:regulated-finite-box-maxwell-trace-section06}

Recall that
\begin{equation}
\label{eq:lambda-mn-section06}
    \lambdamn=\abs{k_m}^2+\nu_n^2,
    \qquad
    k_m=\frac{2\pi}{L}m,
    \qquad
    \nu_n=\frac{\pi n}{a} .
\end{equation}
Define the one-channel scalar Dirichlet regulated energy in the finite box by
\begin{equation}
\label{eq:EDtau-def-section06}
    \EDtau(L,a)
    :=
    \frac{\hbarc}{2}
    \sum_{m\in\Z^2}\sum_{n=1}^{\infty}
    \lambdamn^{1/2}e^{-\tau\lambdamn} .
\end{equation}
Define also the reduced Neumann zero-vertical branch
\begin{equation}
\label{eq:Btau-def-section06}
    \Btau(L)
    :=
    \frac{\hbarc}{2}
    \sum_{m\in\Z^2\setminus\{0\}}
    \abs{k_m}e^{-\tau\abs{k_m}^2} .
\end{equation}
The explicit trace formula of Section~\ref{sec:te-tm-decomposition}, namely
\eqref{eq:regulated-maxwell-trace-explicit-section05}, is then
\begin{equation}
\label{eq:Maxwell-regulated-energy-2D-plus-B-section06}
    \E[\UtauMx(L,a)]
    =
    2\EDtau(L,a)+\Btau(L).
\end{equation}
Equivalently, if
\begin{equation}
\label{eq:ENtau-def-section06}
    \ENtau(L,a)
    :=
    \frac{\hbarc}{2}
    \Tr'\paren{\LN^{1/2}e^{-\tau\LN}},
\end{equation}
then the reduced Neumann identity is
\begin{equation}
\label{eq:ENtau-equals-EDtau-plus-Btau-section06}
    \ENtau(L,a)=\EDtau(L,a)+\Btau(L),
\end{equation}
and
\begin{equation}
\label{eq:Maxwell-regulated-energy-D-plus-N-section06}
    \E[\UtauMx(L,a)]
    =
    \EDtau(L,a)+\ENtau(L,a).
\end{equation}
The term \(\Btau(L)\) is not discarded at the regulated finite-volume level.
It is a genuine Maxwell branch.  Its special role is that it has no dependence
on the plate separation \(a\).

\subsection{The interaction finite part}
\label{subsec:interaction-finite-part-section06}

The quantity computed below is the standard parallel-plate interaction energy
per unit area.  This fixes the finite-part convention used in this section.

\begin{definition}[Parallel-plate interaction finite part]
\label{def:parallel-plate-interaction-finite-part}
For a regulated large-area energy density \(F_\tau(a)\), the notation
\[
    \FP_{\mathrm{int},\,\tau\to0^+} F_\tau(a)
\]
means the following operation: subtract the bulk vacuum terms, the separate
one-plate self-energy terms, and all terms independent of the separation
\(a\), and retain the finite separation-dependent part.  After these standard
bulk and one-plate subtractions have been made, the remaining finite part
determines the finite plate-separation force
\[
    -\frac{\partial}{\partial a}
    \FP_{\mathrm{int},\,\tau\to0^+} F_\tau(a).
\]
Thus an \(a\)-independent regulated density has zero interaction finite part.
\end{definition}

This convention is weaker than, and different from, assigning an absolute
renormalized vacuum energy.  It fixes the physically relevant plate
interaction and leaves no freedom to keep an arbitrary constant independent of
\(a\).

The branch \(\Btau(L)\) is exactly of this removable type.  For fixed
\(\tau>0\), the Riemann-sum limit gives
\begin{align}
\label{eq:Btau-large-area-limit-section06}
    \lim_{L\to\infty}\frac{\Btau(L)}{L^2}
    &=
    \frac{\hbarc}{2}
    \int_{\R^2}\frac{\dd^2 k}{(2\pi)^2}
    \abs{k}e^{-\tau\abs{k}^2}  \\
    &=
    \frac{\hbarc}{2}\cdot\frac{1}{2\pi}
    \int_0^\infty k^2e^{-\tau k^2}\,\dd k \\
    &=
    \frac{\hbarc}{16\sqrt\pi}\,\tau^{-3/2} .
\end{align}
The right-hand side has no dependence on \(a\).  Therefore
\begin{equation}
\label{eq:Btau-FPint-zero-section06}
    \FP_{\mathrm{int},\,\tau\to0^+}
    \left[\lim_{L\to\infty}\frac{\Btau(L)}{L^2}\right]
    =0 .
\end{equation}
Consequently,
\begin{equation}
\label{eq:Maxwell-FP-reduces-to-Dirichlet-section06}
    \FP_{\mathrm{int},\,\tau\to0^+}
    \left[\lim_{L\to\infty}\frac{1}{L^2}\E[\UtauMx(L,a)]\right]
    =
    2\,
    \FP_{\mathrm{int},\,\tau\to0^+}
    \left[\lim_{L\to\infty}\frac{\EDtau(L,a)}{L^2}\right].
\end{equation}
Thus the remaining calculation is the one-channel scalar Dirichlet plate
finite part.

\subsection{The scalar Dirichlet finite part}
\label{subsec:scalar-dirichlet-finite-part-section06}

For completeness, we recall the standard zeta extraction of the one-channel
Dirichlet interaction finite part.  At fixed \(\tau>0\), the large-area limit
of \(L^{-2}\EDtau(L,a)\) is
\begin{equation}
\label{eq:large-area-Dirichlet-heat-regulated-density}
    \mathcal E_{D,\tau}(a)
    =
    \frac{\hbarc}{2}
    \sum_{n=1}^{\infty}
    \int_{\R^2}\frac{\dd^2k}{(2\pi)^2}
    \paren{\abs{k}^2+\paren{\frac{\pi n}{a}}^2}^{1/2}
    e^{-\tau\left(\abs{k}^2+(\pi n/a)^2\right)} .
\end{equation}
The corresponding spectral zeta density is initially defined for
\(\mathrm{Re}\,s>3/2\) by
\begin{equation}
\label{eq:ZD-s-def-section06}
    Z_D(s,a)
    :=
    \sum_{n=1}^{\infty}
    \int_{\R^2}\frac{\dd^2k}{(2\pi)^2}
    \paren{\abs{k}^2+\paren{\frac{\pi n}{a}}^2}^{-s} .
\end{equation}
Equivalently, the energy zeta regularization may be written, for
\(\mathrm{Re}\,s>2\), as
\begin{equation}
\label{eq:ED-zeta-regularized-section06}
    \mathcal E_D(s,a)
    :=
    \frac{\hbarc}{2}\,\mu^{2s}
    \sum_{n=1}^{\infty}
    \int_{\R^2}\frac{\dd^2k}{(2\pi)^2}
    \paren{\abs{k}^2+\paren{\frac{\pi n}{a}}^2}^{1/2-s} .
\end{equation}
Here \(\mu\) is an arbitrary reference scale inserted only to keep dimensions
fixed away from \(s=0\).  The standard dimensional integral gives
\begin{equation}
\label{eq:dimensional-integral-section06}
    \int_{\R^2}\frac{\dd^2k}{(2\pi)^2}
    \paren{\abs{k}^2+M^2}^{1/2-s}
    =
    \frac{1}{4\pi}
    \frac{\Gamma(s-3/2)}{\Gamma(s-1/2)}
    M^{3-2s},
\end{equation}
by analytic continuation from the convergent half-plane.  Therefore
\begin{equation}
\label{eq:ED-zeta-evaluated-section06}
    \mathcal E_D(s,a)
    =
    \frac{\hbarc}{8\pi}
    \mu^{2s}
    \frac{\Gamma(s-3/2)}{\Gamma(s-1/2)}
    \paren{\frac{\pi}{a}}^{3-2s}
    \zeta(2s-3).
\end{equation}
This meromorphic expression is regular at \(s=0\).  Since
\begin{equation}
\label{eq:gamma-zeta-values-section06}
    \frac{\Gamma(-3/2)}{\Gamma(-1/2)}=-\frac23,
    \qquad
    \zeta(-3)=\frac{1}{120},
\end{equation}
one obtains
\begin{equation}
\label{eq:Dirichlet-finite-part-result-section06}
    \mathcal E_D(0,a)
    =
    -\frac{\pi^2\hbarc}{1440a^3}.
\end{equation}
This is the standard one-channel scalar Dirichlet parallel-plate interaction
finite part:
\begin{equation}
\label{eq:Dirichlet-FPint-density-section06}
    \FP_{\mathrm{int},\,\tau\to0^+}\mathcal E_{D,\tau}(a)
    =
    -\frac{\pi^2\hbarc}{1440a^3}.
\end{equation}

\begin{remark}[Heat regulator and zeta extraction]
\label{rem:heat-zeta-compatibility-section06}
The stochastic representation uses the heat cutoff
\(\lambda^{1/2}e^{-\tau\lambda}\), whereas
\eqref{eq:ED-zeta-regularized-section06} uses analytic continuation.  These
extract the same interaction finite part.  Indeed, Mellin inversion gives,
for \(\gamma\) in the convergent half-plane,
\begin{equation}
\label{eq:Mellin-heat-zeta-section06}
    \mathcal E_{D,\tau}(a)
    =
    \frac{\hbarc}{2}\frac{1}{2\pi\ii}
    \int_{\gamma-\ii\infty}^{\gamma+\ii\infty}
    \Gamma(z)\tau^{-z}Z_D(z-1/2,a)\,\dd z .
\end{equation}
Shifting the contour gives the short-time expansion.  The poles away from
\(z=0\) produce the local divergent terms.  The residue at \(z=0\) from
\(\Gamma(z)\) is exactly \((\hbarc/2)Z_D(-1/2,a)
=\mathcal E_D(0,a)\), which is the finite separation-dependent term displayed
in \eqref{eq:Dirichlet-finite-part-result-section06}.  This is the sense in
which \eqref{eq:Dirichlet-FPint-density-section06} is a heat finite part and
not a change of regulator.
\end{remark}

\subsection{The reduced Neumann finite part}
\label{subsec:reduced-neumann-finite-part-section06}

The reduced Neumann channel differs from the Dirichlet channel only by the
zero-vertical branch \(\Btau(L)\).  Combining
\eqref{eq:ENtau-equals-EDtau-plus-Btau-section06} with
\eqref{eq:Btau-FPint-zero-section06} gives
\begin{align}
\label{eq:Neumann-FPint-equals-Dirichlet-section06}
    \FP_{\mathrm{int},\,\tau\to0^+}
    \left[\lim_{L\to\infty}\frac{\ENtau(L,a)}{L^2}\right]
    &=
    \FP_{\mathrm{int},\,\tau\to0^+}
    \left[\lim_{L\to\infty}\frac{\EDtau(L,a)}{L^2}\right] \\
    &=
    -\frac{\pi^2\hbarc}{1440a^3} .
\end{align}
Thus the reduced scalar Neumann sector has the same interaction finite part as
the scalar Dirichlet sector.  The equality is not a statement that the raw
regulated traces are identical; their difference is the real branch
\(\Btau(L)\), which is independent of \(a\).

\subsection{The Maxwell plate result}
\label{subsec:maxwell-plate-result-section06}

We now combine the preceding identities.  From
\eqref{eq:Maxwell-regulated-energy-D-plus-N-section06}, the Maxwell interaction
finite part per unit area is the sum of the Dirichlet and reduced Neumann
interaction finite parts.  Therefore,
\begin{align}
\label{eq:Maxwell-FPint-final-calculation-section06}
    \FP_{\mathrm{int},\,\tau\to0^+}
    \left[\lim_{L\to\infty}\frac{1}{L^2}\E[\UtauMx(L,a)]\right]
    &=
    -\frac{\pi^2\hbarc}{1440a^3}
    -\frac{\pi^2\hbarc}{1440a^3} \\
    &=
    -\frac{\pi^2\hbarc}{720a^3} .
\end{align}

\begin{theorem}[Maxwell parallel-plate finite part]
\label{thm:Maxwell-parallel-plate-finite-part}
Let \(\LMx\) be the reduced Maxwell operator on
\[
    \OmegaLa=\TtwoL\times[0,a],
\]
with perfect-conductor boundary conditions: it is defined on divergence-free
electric fields with \(n\times E=0\) on the plates, and the static normal zero
mode is removed.  Let \(\Sigmatau\) be the
heat-regularized prescribed Maxwell Gaussian source of
Section~\ref{sec:riesz-gaussian-source}.  Then, under the standard
parallel-plate interaction finite-part prescription of
Definition~\ref{def:parallel-plate-interaction-finite-part},
all objects in the expectation below are the finite-volume objects associated
with \(\Omega_{L,a}\); in particular \(\LMx=\LMx(L,a)\) and
\(\Sigmatau=\Sigmatau(L,a)\):
\begin{equation}
\label{eq:Maxwell-main-finite-part-theorem-section06}
    \boxed{
    \FP_{\mathrm{int},\,\tau\to0^+}
    \left[\lim_{L\to\infty}\frac{1}{L^2}
    \E\left[
        \frac12\inner{\Sigmatau}{g\LMx^{-1}\Sigmatau}
    \right]
    \right]
    =
    -\frac{\pi^2\hbarc}{720a^3}.
    }
\end{equation}
\end{theorem}

\begin{proof}
For finite \(L\), Theorem~\ref{thm:regulated-maxwell-quadratic-trace} gives
\begin{equation*}
    \E\left[
        \frac12\inner{\Sigmatau}{g\LMx^{-1}\Sigmatau}
    \right]
    =
    \frac{\hbarc}{2}\Trphys\paren{\LMx^{1/2}e^{-\tau\LMx}} .
\end{equation*}
Theorem~\ref{thm:maxwell-trace-Dirichlet-Neumann} identifies this trace with
Dirichlet plus reduced Neumann scalar traces.  Equations
\eqref{eq:Dirichlet-FPint-density-section06} and
\eqref{eq:Neumann-FPint-equals-Dirichlet-section06} evaluate their interaction
finite parts.  Adding them gives \eqref{eq:Maxwell-main-finite-part-theorem-section06}.
\end{proof}

\begin{remark}[Dependence of the calculation]
\label{rem:what-used-section06}
The factor of two in \eqref{eq:Maxwell-FPint-final-calculation-section06}
comes from the Maxwell spectral identity
\(\LMx\simeq \LD\oplus\LNp\), not from assuming two independent scalar copies.
The reduced Neumann zero-vertical branch was retained in the regulated trace
and removed only at the interaction finite-part stage because it is
\(a\)-independent.  The broader interpretation and limitations of the theorem
are collected in Section~\ref{sec:discussion-limitations}.
\end{remark}
\section{Discussion and limitations}
\label{sec:discussion-limitations}

The preceding sections prove a Maxwell version of the scalar
codimension-three Riesz/Gaussian representation in the parallel-plate geometry.
The result is narrow but complete within that scope.  Starting from the
physical perfect-conductor Maxwell operator on divergence-free electric fields,
we obtain a finite-volume Riesz-reduced Green operator, a heat-regularized
Gaussian source with prescribed covariance, a regulated quadratic-form trace
identity, an exact TE/TM trace decomposition, and finally the standard
parallel-plate interaction finite part
\[
    -\frac{\pi^2\hbarc}{720a^3} .
\]

The final coefficient is the standard electromagnetic parallel-plate value.
The contribution of the theorem is the operator-theoretic representation that
produces this value from the reduced Maxwell Hilbert space for perfectly
conducting plates.
The factor of two in the final coefficient is not inserted by hand.  It comes
from the spectral equivalence
\[
    \LMx\simeq\LD\oplus\LNp,
\]
proved in Theorem~\ref{thm:maxwell-trace-Dirichlet-Neumann}.  The reduced
Neumann zero-vertical branch is present in the finite-volume Maxwell trace and
is discarded only at the interaction finite-part stage because it is
independent of the plate separation.
Read together with \cite{KhanKhan2026Scalar}, the present paper separates the
program into two levels.  The scalar preprint isolates the codimension-three
Riesz/Gaussian trace mechanism and records the scalar Dirichlet plate
benchmark.  The present paper shows that the same mechanism extends to the
genuine Maxwell parallel-plate operator once the physical Hilbert space,
perfect-conductor boundary condition, static zero mode, and TE/TM decomposition
are treated directly.

\subsection{What the theorem establishes}
\label{subsec:what-the-theorem-establishes}

The first established point is the operator construction.  The Maxwell operator
is not introduced by a formal differential expression alone.  It is defined by
the closed form
\[
    \qMx[E,F]
    =
    \int_{\OmegaLa}
    (\nabla\times E)\cdot\overline{(\nabla\times F)}\,\dd x
\]
on the reduced physical form domain \(\Vphys\).  The boundary condition
\(n\times E=0\) and the constraint \(\nabla\cdot E=0\) are built into this
form domain.  The static normal field is removed before the inverse is used.
Thus the operator entering the Riesz identity is a positive self-adjoint
finite-volume Maxwell operator with a genuine spectral gap.

The second established point is spectral admissibility.  The mixed sine/cosine
Fourier characterization of the form domain shows that the spectrum is
discrete in finite volume and that
\[
    \Tr\paren{\LMx^p e^{-\tau\LMx}}<\infty
    \qquad(\tau>0,
    \ p\geq0).
\]
This justifies the heat-regularized Gaussian source as an honest Hilbert-valued
random variable and makes the covariance trace-class.

The third established point is the codimension-three Riesz reduction on the
Maxwell Hilbert space.  Since the proof uses only the spectral theorem for a
positive self-adjoint operator with a lower bound, it applies to \(\LMx\) just
as it applies to a scalar operator:
\[
    \int_{\R^3}\frac{\dd^3q}{(2\pi)^3}
    (\LMx+\abs q^2)^{-5/2}
    =
    \frac{1}{6\pi^2}\LMx^{-1} .
\]
With \(g=\kappa/(6\pi^2)\), this gives the finite-volume mediator
\(\VMx=g\LMx^{-1}\).

The fourth established point is the stochastic trace identity.  For
\[
    \Sigmatau
    =
    \paren{\frac{\hbarc}{g}}^{1/2}
    \LMx^{3/4}e^{-\tau\LMx/2}\XiNoise,
\]
the covariance is
\[
    \Ctau
    =
    \frac{\hbarc}{g}\LMx^{3/2}e^{-\tau\LMx} .
\]
Pairing this covariance with the Green operator \(g\LMx^{-1}\) gives
\[
    g\LMx^{-1}\Ctau
    =
    \hbarc\,\LMx^{1/2}e^{-\tau\LMx},
\]
and therefore
\[
    \E[\UtauMx]
    =
    \frac{\hbarc}{2}
    \Trphys\paren{\LMx^{1/2}e^{-\tau\LMx}} .
\]

The fifth established point is the plate finite part.  The Maxwell trace is
Dirichlet plus reduced Neumann.  The reduced Neumann trace differs from the
Dirichlet trace by the zero-vertical branch
\[
    \Btau(L)
    =
    \frac{\hbarc}{2}
    \sum_{m\ne0}\abs{k_m}e^{-\tau\abs{k_m}^2},
\]
which carries no \(a\)-dependence.  Under the standard parallel-plate
interaction finite-part prescription, this branch contributes zero.  The two
remaining scalar interaction finite parts are equal, each giving
\(-\pi^2\hbarc/(1440a^3)\), and their sum gives the electromagnetic coefficient.

\subsection{Why the electric-field formulation was used}
\label{subsec:why-electric-field-formulation}

The main proof avoids vector-potential gauge issues by working directly with
physical electric fields.  Vector-potential and covariant gauge-fixed
formulations may also be used, but the present choice of variables makes the
Hilbert-space statement precise.  Nonzero Maxwell modes may be related to
temporal/Coulomb-gauge vector-potential modes, but that relationship is not
needed for the proof of the spectral trace identity.
Appendix~\ref{app:optional-gauge-comment} records the limited gauge-theoretic
interpretation and explicitly separates it from a full covariant determinant
calculation.

The important consequence is that no Faddeev--Popov cancellation is hidden in
the main argument.  The trace is taken directly over the physical electric-field
Hilbert space after the static zero mode has been removed.  If one chooses a
covariant gauge-fixed formulation instead, one must independently show that the
vector and ghost determinants reduce to the same physical trace.  The present
paper does not use that determinant formulation; it works directly with the
physical electric-field trace.

\subsection{What the theorem does not establish}
\label{subsec:what-the-theorem-does-not-establish}

The covariance is prescribed.  The construction does not show that
\[
    \Ctau=\frac{\hbarc}{g}\LMx^{3/2}e^{-\tau\LMx}
\]
is forced by microscopic electrodynamics.  It shows that if this
heat-regularized covariance is assigned on the reduced Maxwell Hilbert space,
then the expected quadratic Green energy is exactly the heat-regularized
Maxwell trace.  Deriving the covariance from a photon path integral, a
worldline representation, or an effective action with matter fields would be a
different project.

The ambient product operator is also not a local higher-dimensional Maxwell
operator.  It is the product spectral operator
\[
    \Lamb=\LMx\otimes I+I\otimes(-\Delta_y)
\]
used to define a fractional Riesz-type mediator.  The transverse restriction is
implemented through a transverse momentum integral, or equivalently through a
mollified transverse distributional limit.  No point-evaluation map on
\(L^2(\R^3_y)\) is used, and no ordinary six-dimensional Maxwell propagator is
claimed.

The result is also not a statement about arbitrary conductor geometries.  The
proof uses the product slab \(\TtwoL\times[0,a]\), lateral Fourier modes, and
the explicit sine/cosine decomposition induced by flat plates.  Non-planar
geometries require new spectral analysis.  In particular, conducting spherical
shells, curved boundaries, and edge geometries are outside the theorem.

Finally, this paper does not use or justify any reference-cell calibration,
cube extremality argument, loop correction, running coupling, or
fundamental-constant identification.  Those topics would require additional
normalizations or additional physics not present in the theorem proved here.
Those calibration questions are treated separately in
\cite{KhanKhan2026Scalar} and are not part of the Maxwell theorem proved here.

\subsection{Order of limits and finite parts}
\label{subsec:order-of-limits-discussion}

The order of operations is essential:
\[
\begin{aligned}
    L<\infty
    &\quad\longrightarrow\quad
    \text{regulated operator identity}
    \quad\longrightarrow\quad
    \text{trace density at fixed }\tau \\
    &\quad\longrightarrow\quad
    L\to\infty
    \quad\longrightarrow\quad
    \FP_{\mathrm{int},\,\tau\to0^+}.
\end{aligned}
\]
For finite \(L\), the reduced operator has a bounded inverse.  In the
large-area limit, the spectral gap collapses because the zero-vertical branch
has eigenvalues \(\abs{k_m}^2\) with \(k_m\to0\).  Thus it would be incorrect
to treat \(\LMx^{-1}\) as a bounded infinite-area operator.  The theorem avoids
this by proving the stochastic identity in finite volume and taking the
large-area limit only at the level of the heat-regularized trace density.  The
interaction finite part is then applied to that large-area density.

The finite part used in the final result is the standard interaction finite
part.  It removes bulk terms, one-plate self-energies, and terms independent of
the separation \(a\), retaining the finite separation-dependent interaction
energy.  This convention is not optional bookkeeping: it is what permits the
real but \(a\)-independent branch \(\Btau(L)\) to be discarded in the
interaction energy without discarding it from the finite-volume Maxwell trace.

\subsection{Beyond parallel plates}
\label{subsec:possible-next-tests}

Non-planar conductor geometries require separate spectral analysis.  In such
geometries the physical Maxwell Hilbert space, boundary conditions, zero modes,
and TE/TM or vector-spherical spectral structure must be treated anew.  The
present result should therefore be read as a parallel-plate theorem.

Thus the correct conclusion is deliberately modest.  The parallel-plate Maxwell
operator passes the Riesz/Gaussian representation test, and the standard
interaction finite part is the electromagnetic plate coefficient.  No broader
geometric universality or microscopic QED derivation follows from this result
alone.

\subsection{Conclusion}
\label{subsec:conclusion}

The theorem proved in this paper can be summarized as follows.  On the reduced
physical Maxwell Hilbert space for perfectly conducting parallel plates, the
codimension-three transverse reduction produces the Green operator
\(g\LMx^{-1}\).  With the prescribed heat-regularized covariance
\((\hbarc/g)\LMx^{3/2}e^{-\tau\LMx}\), the expected quadratic Green energy is
exactly
\[
    \frac{\hbarc}{2}
    \Trphys\paren{\LMx^{1/2}e^{-\tau\LMx}} .
\]
The physical Maxwell trace is \(\LD\oplus\LNp\), and the standard
parallel-plate interaction finite part is
\[
    -\frac{\pi^2\hbarc}{720a^3} .
\]
The relation to the companion scalar preprint \cite{KhanKhan2026Scalar} is
therefore straightforward: that paper established the scalar codimension-three
Riesz/Gaussian representation scheme and its one-channel Dirichlet
parallel-plate specialization, while the present paper extends the same scheme
to the reduced physical Maxwell operator and derives the electromagnetic plate
coefficient from the actual Maxwell spectrum.
This completes the Maxwell parallel-plate version of the codimension-three
Riesz/Gaussian trace representation.

\appendix
\section{Details of the Fourier coefficient-domain lemma}
\label{app:fourier-domain-lemma}

This appendix gives a detailed proof of the coefficient-domain statement used in
Section~\ref{sec:fourier-domain-heat-trace}.  The point is to connect the
closed-form definition of the Maxwell operator in Section~\ref{sec:maxwell-operator}
with the separated Fourier calculation.  No general Maxwell compactness theorem
is used here; the proof is specific to the finite periodic slab
\[
    \OmegaLa=\TtwoL\times[0,a].
\]

The convention from Section~\ref{sec:maxwell-operator} is important:
\(\Hzcurl(\OmegaLa;\C^3)\) denotes the closure, in the
\(H(\operatorname{curl})\) graph norm, of smooth lateral-periodic fields whose
tangential trace vanishes on the two plates.  Thus the perfect-conductor
condition is imposed at the form-domain level.

\subsection{Mixed Fourier expansions}
\label{app:subsec:mixed-expansions}

Let
\[
    k_m:=\frac{2\pi}{L}m,
    \qquad m\in\Z^2,
    \qquad
    \nu_n:=\frac{\pi n}{a},
    \qquad n\in\N_0,
\]
and define
\[
    \phi_m(r):=L^{-1}\ee^{\ii k_m\cdot r},
    \qquad r=(x,y)\in\TtwoL .
\]
For the plate direction set
\[
    s_n(z):=\sqrt{\frac2a}\sin(\nu_n z),\qquad n\ge1,
\]
and
\[
    c_0(z):=a^{-1/2},
    \qquad
    c_n(z):=\sqrt{\frac2a}\cos(\nu_n z),\qquad n\ge1.
\]
The lateral modes \(\{\phi_m\}_{m\in\Z^2}\) are an orthonormal basis of
\(L^2(\TtwoL)\), the sine modes \(\{s_n\}_{n\ge1}\) are an orthonormal basis of
\(L^2(0,a)\), and the cosine modes \(\{c_n\}_{n\ge0}\) are another orthonormal
basis of \(L^2(0,a)\).

Consequently every \(L^2\) vector field has the mixed expansion
\begin{align}
    E_{\parallel}(r,z)
    &=
    \sum_{m\in\Z^2}\sum_{n\ge1}
    a_{m,n}\,\phi_m(r)s_n(z),
    \qquad a_{m,n}\in\C^2,                                  \label{eq:app-mixed-tangential}\\
    E_z(r,z)
    &=
    \sum_{m\in\Z^2}\sum_{n\ge0}
    b_{m,n}\,\phi_m(r)c_n(z),
    \qquad b_{m,n}\in\C .                                     \label{eq:app-mixed-normal}
\end{align}
Here \(E_{\parallel}=(E_x,E_y)\).  Parseval gives
\begin{equation}
\label{eq:app-parseval-L2}
    \norm{E}_{L^2(\OmegaLa)}^2
    =
    \sum_{m\in\Z^2}\sum_{n\ge1}\abs{a_{m,n}}^2
    +
    \sum_{m\in\Z^2}\sum_{n\ge0}\abs{b_{m,n}}^2 .
\end{equation}
The use of sine modes for the tangential components is compatible with the
condition \(E_x=E_y=0\) on \(z=0,a\) for finite trigonometric sums.  For general
\(L^2\) fields the boundary condition is not read pointwise from the series; it
is read through the graph-norm closure defining \(\Hzcurl\).

For \(n\ge1\), put
\begin{equation}
\label{eq:app-lambda-mn}
    \lambda_{m,n}:=\abs{k_m}^2+\nu_n^2 .
\end{equation}

\subsection{The coefficient spaces}
\label{app:subsec:coefficient-spaces}

Define \(\mathfrak C_{\mathrm{PEC}}\) to be the space of coefficient families
\((a,b)\) satisfying
\begin{equation}
\label{eq:app-CPEC-summability}
    \sum_{m\in\Z^2}\sum_{n\ge1}
    (1+\lambda_{m,n})
    \paren{\abs{a_{m,n}}^2+\abs{b_{m,n}}^2}
    +
    \sum_{m\in\Z^2}(1+\abs{k_m}^2)\abs{b_{m,0}}^2
    <\infty
\end{equation}
and, for all \(m\in\Z^2\) and \(n\ge1\),
\begin{equation}
\label{eq:app-div-constraint}
    \ii k_m\cdot a_{m,n}-\nu_n b_{m,n}=0 .
\end{equation}
The reduced physical coefficient space is
\begin{equation}
\label{eq:app-Cphys}
    \mathfrak C_{\mathrm{phys}}
    :=
    \set{(a,b)\in\mathfrak C_{\mathrm{PEC}}: b_{0,0}=0} .
\end{equation}

\begin{lemma}[Detailed Fourier characterization]
\label{lem:app-fourier-characterization}
The reconstruction map \eqref{eq:app-mixed-tangential}--\eqref{eq:app-mixed-normal}
identifies \(\mathfrak C_{\mathrm{PEC}}\) with
\[
    \VPEC=
    \set{E\in\Hzcurl(\OmegaLa;\C^3):\nabla\cdot E=0
    \text{ in }\mathcal D'(\OmegaLa)} .
\]
For \(E\in\VPEC\), the curl norm is
\begin{equation}
\label{eq:app-curl-norm-diagonal}
    \norm{\nabla\times E}_{L^2}^2
    =
    \sum_{m\in\Z^2}\sum_{n\ge1}
    \lambda_{m,n}
    \paren{\abs{a_{m,n}}^2+\abs{b_{m,n}}^2}
    +
    \sum_{m\in\Z^2}\abs{k_m}^2\abs{b_{m,0}}^2 .
\end{equation}
Moreover, imposing \(b_{0,0}=0\) identifies \(\mathfrak C_{\mathrm{phys}}\) with
\[
    \Vphys=\VPEC\cap\Hphys .
\]
\end{lemma}

\begin{proof}
We split the proof into four elementary steps.

\emph{Step 1: divergence coefficients.}
For a finite trigonometric field of the form
\eqref{eq:app-mixed-tangential}--\eqref{eq:app-mixed-normal}, direct
differentiation gives
\begin{equation}
\label{eq:app-div-expansion}
    \nabla\cdot E
    =
    \sum_{m\in\Z^2}\sum_{n\ge1}
    \paren{\ii k_m\cdot a_{m,n}-\nu_n b_{m,n}}
    \phi_m(r)s_n(z).
\end{equation}
There is no contribution from \(b_{m,0}\), since \(c_0'(z)=0\).  Thus the
constraint \eqref{eq:app-div-constraint} is exactly the Fourier form of
\(\nabla\cdot E=0\).

\emph{Step 2: curl coefficients and the finite-dimensional identity.}
For \(n\ge1\), the curl coefficients are
\begin{align}
\label{eq:app-curl-coeffs-positive-n}
    C^x_{m,n}&=\ii k_{m,y}b_{m,n}-\nu_n a^y_{m,n},\nonumber\\
    C^y_{m,n}&=\nu_n a^x_{m,n}-\ii k_{m,x}b_{m,n},\\
    C^z_{m,n}&=\ii\paren{k_{m,x}a^y_{m,n}-k_{m,y}a^x_{m,n}}.\nonumber
\end{align}
For the normal branch \(n=0\),
\begin{equation}
\label{eq:app-curl-coeffs-zero-n}
    C^x_{m,0}=\ii k_{m,y}b_{m,0},
    \qquad
    C^y_{m,0}=-\ii k_{m,x}b_{m,0},
    \qquad
    C^z_{m,0}=0 .
\end{equation}
For \(n\ge1\), let
\[
    D_{m,n}:=\ii k_m\cdot a_{m,n}-\nu_n b_{m,n} .
\]
A direct expansion gives the identity
\begin{align}
\label{eq:app-curl-div-identity}
    &\abs{C^x_{m,n}}^2+
     \abs{C^y_{m,n}}^2+
     \abs{C^z_{m,n}}^2+
     \abs{D_{m,n}}^2        \\
    &\hspace{3cm}=
     \lambda_{m,n}
     \paren{\abs{a_{m,n}}^2+\abs{b_{m,n}}^2} .\nonumber
\end{align}
One way to check the cancellation is to observe that the terms involving
\(b_{m,n}\overline{a^x_{m,n}}\) and \(b_{m,n}\overline{a^y_{m,n}}\) occur with
opposite signs in the curl part and in \(\abs{D_{m,n}}^2\).  The remaining
terms use the two-dimensional identity
\[
    \abs{k_m\cdot a_{m,n}}^2+
    \abs{k_{m,x}a^y_{m,n}-k_{m,y}a^x_{m,n}}^2
    =
    \abs{k_m}^2\abs{a_{m,n}}^2,
\]
valid for complex \(a_{m,n}\in\C^2\) and real \(k_m\in\R^2\).

\emph{Step 3: from coefficients to the form domain.}
Let \((a,b)\in\mathfrak C_{\mathrm{PEC}}\).  Let \(E^{(N)}\) be any increasing
sequence of finite rectangular truncations of
\eqref{eq:app-mixed-tangential}--\eqref{eq:app-mixed-normal}.  Each truncation is
smooth, lateral-periodic, and has tangential components built from sine modes,
so
\[
    n\times E^{(N)}=0
    \qquad\text{on }z=0,a .
\]
The coefficient constraint \eqref{eq:app-div-constraint} is preserved under
truncation, hence \(\nabla\cdot E^{(N)}=0\).  By
\eqref{eq:app-curl-div-identity} with \(D_{m,n}=0\) and by
\eqref{eq:app-curl-coeffs-zero-n}, the summability condition
\eqref{eq:app-CPEC-summability} implies that \(E^{(N)}\) is Cauchy in the
\(H(\operatorname{curl})\) graph norm.  Since \(\Hzcurl\) is the graph-norm
closure of smooth PEC fields, the limit belongs to \(\Hzcurl\).  The divergence
constraint passes to the limit distributionally.  Therefore the reconstructed
field belongs to \(\VPEC\).  This proves
\[
    \mathfrak C_{\mathrm{PEC}}\subset\VPEC .
\]

\emph{Step 4: from the form domain to coefficients.}
Conversely, let \(E\in\VPEC\).  Its mixed Fourier coefficients are defined by
\eqref{eq:app-mixed-tangential}--\eqref{eq:app-mixed-normal}.  Since
\(\nabla\cdot E=0\) in distributions, testing against
\(\overline{\phi_m}s_n\), \(n\ge1\), gives precisely
\eqref{eq:app-div-constraint}.  No endpoint term arises in this test because
\(s_n(0)=s_n(a)=0\).

The weak curl \(\nabla\times E\) is in \(L^2\).  Its Fourier coefficients are the
distributional derivatives displayed in
\eqref{eq:app-curl-coeffs-positive-n}--\eqref{eq:app-curl-coeffs-zero-n}.  These
formulas are obtained by differentiating finite partial sums, testing against
the mixed sine/cosine modes, and passing to the distributional limit.  Since
\(\nabla\times E\in L^2\), uniqueness of Fourier coefficients identifies the
\(L^2\)-Fourier coefficients of the weak curl with the displayed expressions.
Parseval
therefore gives
\[
    \sum\abs{C}^2<\infty .
\]
Using \eqref{eq:app-curl-div-identity} and the fact that
\(D_{m,n}=0\), we obtain
\[
    \sum_{m\in\Z^2}\sum_{n\ge1}
    \lambda_{m,n}
    \paren{\abs{a_{m,n}}^2+\abs{b_{m,n}}^2}
    +
    \sum_{m\in\Z^2}\abs{k_m}^2\abs{b_{m,0}}^2
    <\infty .
\]
Together with the \(L^2\) Parseval identity \eqref{eq:app-parseval-L2}, this is
exactly \eqref{eq:app-CPEC-summability}.  Hence the coefficient family of
\(E\) lies in \(\mathfrak C_{\mathrm{PEC}}\), and the curl norm is
\eqref{eq:app-curl-norm-diagonal}.  Thus
\[
    \VPEC\subset\mathfrak C_{\mathrm{PEC}} .
\]
The two inclusions prove the identification.

Finally, the mode \(b_{0,0}\) reconstructs the vector field
\[
    b_{0,0}\phi_0(r)c_0(z)e_z
    =
    b_{0,0}(L^2a)^{-1/2}e_z
    =
    b_{0,0}\hzero .
\]
Thus the condition \(b_{0,0}=0\) is exactly orthogonality to the static normal
mode \(\hzero\).  This proves the reduced statement.
\end{proof}

\subsection{Immediate consequences}
\label{app:subsec:immediate-consequences}

The appendix proof gives two useful consequences used repeatedly in the main
text.

\begin{corollary}[Density of finite transverse PEC sums]
\label{cor:app-density-finite-transverse-sums}
Finite trigonometric sums satisfying the coefficient constraint
\eqref{eq:app-div-constraint} are dense in \(\VPEC\) with respect to the norm
\[
    \norm{E}_{L^2}^2+\norm{\nabla\times E}_{L^2}^2 .
\]
After imposing \(b_{0,0}=0\), they are dense in \(\Vphys\).
\end{corollary}

\begin{proof}
For \((a,b)\in\mathfrak C_{\mathrm{PEC}}\), finite rectangular truncations remain
constraint-satisfying and converge in the graph norm by the weighted
summability condition and \eqref{eq:app-curl-norm-diagonal}.  The reduced case
is the same, with the coefficient \(b_{0,0}\) omitted.
\end{proof}

\begin{corollary}[Zero mode in the coefficient model]
\label{cor:app-zero-mode-coefficient-model}
The only coefficient family in \(\mathfrak C_{\mathrm{PEC}}\) with zero curl norm
is the family with possibly nonzero \(b_{0,0}\) and all other coefficients zero.
Thus the unreduced zero-energy space is \(\operatorname{span}\{\hzero\}\).
\end{corollary}

\begin{proof}
If the curl norm vanishes, \eqref{eq:app-curl-norm-diagonal} forces all
coefficients with positive weights to vanish.  The only remaining coefficient is
\(b_{0,0}\), which reconstructs \(b_{0,0}\hzero\).
\end{proof}

\begin{remark}[Role of the appendix]
\label{rem:app-role-fourier-domain-lemma}
The diagonal formula \eqref{eq:app-curl-norm-diagonal} is not a separate spectral
assumption.  It follows from the closed PEC form domain and the Fourier
characterization above.  Consequently the later eigenvalue sums for \(\LMx\) are
sums for the self-adjoint operator defined by the Maxwell quadratic form, not a
formal mode-counting replacement for that operator.
\end{remark}
\section{Heat cutoff, zeta extraction, and the plate finite part}
\label{app:heat-zeta-finite-part}

The stochastic identity in the main text uses the heat cutoff
\[
    \lambda^{1/2}e^{-\tau\lambda},
\]
whereas the scalar parallel-plate coefficient is often recalled by analytic
continuation of a spectral zeta function; see, for example,
\cite{Bordag2009,Kirsten2001,Vassilevich2003}.  This appendix records the
explicit connection for the large-area scalar Dirichlet plate trace used in
Section~\ref{sec:finite-part-plate-result}.  The conclusion is that the
standard interaction finite part of the heat-cutoff trace is the same finite
separation-dependent term obtained by zeta extraction:
\[
    -\frac{\pi^2\hbarc}{1440a^3}
\]
for one scalar Dirichlet channel.  The reduced scalar Neumann channel has the
same interaction finite part; its difference from the Dirichlet channel is an
\(a\)-independent zero-vertical branch.

\subsection{The large-area scalar Dirichlet trace}
\label{subsec:app-large-area-dirichlet-trace}

Let the plate separation be \(a>0\).  In the large-area limit, the one-channel
Dirichlet heat-regulated energy density is
\begin{equation}
\label{eq:app-EDtau-density}
    \mathcal E_{D,\tau}(a)
    :=
    \frac{\hbarc}{2}
    \sum_{n=1}^{\infty}
    \int_{\R^2}\frac{\dd^2k}{(2\pi)^2}
    \paren{\abs{k}^2+\paren{\frac{\pi n}{a}}^2}^{1/2}
    e^{-\tau\left(\abs{k}^2+(\pi n/a)^2\right)} .
\end{equation}
For \(\Re w>3/2\), define the scalar Dirichlet spectral zeta density
\begin{equation}
\label{eq:app-ZD-def}
    Z_D(w,a)
    :=
    \sum_{n=1}^{\infty}
    \int_{\R^2}\frac{\dd^2k}{(2\pi)^2}
    \paren{\abs{k}^2+\paren{\frac{\pi n}{a}}^2}^{-w} .
\end{equation}
The heat trace \eqref{eq:app-EDtau-density} is related to \(Z_D\) by Mellin
inversion.

\begin{proposition}[Mellin representation of the heat cutoff]
\label{prop:app-mellin-heat-zeta}
Let \(\gamma>2\).  Then for every \(\tau>0\),
\begin{equation}
\label{eq:app-mellin-heat-zeta}
    \mathcal E_{D,\tau}(a)
    =
    \frac{\hbarc}{2}\,
    \frac{1}{2\pi\ii}
    \int_{\gamma-\ii\infty}^{\gamma+\ii\infty}
    \Gamma(z)\tau^{-z}Z_D\paren{z-\frac12,a}\,\dd z .
\end{equation}
\end{proposition}

\begin{proof}
For \(\gamma>2\), the exponent \(z-1/2\) has real part greater than \(3/2\),
so the defining sum and integral for \(Z_D(z-1/2,a)\) are absolutely
convergent on the vertical line.  The scalar Mellin inversion formula
\[
    e^{-\tau\lambda}
    =
    \frac{1}{2\pi\ii}
    \int_{\gamma-\ii\infty}^{\gamma+\ii\infty}
    \Gamma(z)\tau^{-z}\lambda^{-z}\,\dd z
    \qquad(\lambda>0)
\]
gives
\[
    \lambda^{1/2}e^{-\tau\lambda}
    =
    \frac{1}{2\pi\ii}
    \int_{\gamma-\ii\infty}^{\gamma+\ii\infty}
    \Gamma(z)\tau^{-z}\lambda^{1/2-z}\,\dd z .
\]
Absolute convergence on the chosen line, together with exponential decay of
\(\Gamma(\gamma+\ii t)\) in vertical strips, justifies interchanging the contour
integral with the sum over \(n\) and the integral over \(k\).  This gives
\eqref{eq:app-mellin-heat-zeta}.
\end{proof}

\subsection{Analytic continuation of the Dirichlet zeta density}
\label{subsec:app-ZD-continuation}

The two-dimensional momentum integral gives the meromorphic continuation
explicitly.

\begin{proposition}[Dirichlet zeta density]
\label{prop:app-ZD-continuation}
The function \(Z_D(w,a)\) admits the meromorphic continuation
\begin{equation}
\label{eq:app-ZD-continuation}
    Z_D(w,a)
    =
    \frac{1}{4\pi}
    \frac{\Gamma(w-1)}{\Gamma(w)}
    \paren{\frac{\pi}{a}}^{2-2w}
    \zeta(2w-2).
\end{equation}
Equivalently,
\begin{equation}
\label{eq:app-ZD-z-shifted}
    Z_D\paren{z-\frac12,a}
    =
    \frac{1}{4\pi}
    \frac{1}{z-3/2}
    \paren{\frac{\pi}{a}}^{3-2z}
    \zeta(2z-3).
\end{equation}
\end{proposition}

\begin{proof}
For \(\Re w>3/2\), the standard dimensional integral gives
\begin{equation}
\label{eq:app-dimensional-integral}
    \int_{\R^2}\frac{\dd^2k}{(2\pi)^2}
    \paren{\abs{k}^2+M^2}^{-w}
    =
    \frac{1}{4\pi}\frac{\Gamma(w-1)}{\Gamma(w)}M^{2-2w}.
\end{equation}
Substituting \(M=\pi n/a\) and summing over \(n\ge1\) gives
\eqref{eq:app-ZD-continuation} in the convergent half-plane.  The right-hand
side is meromorphic in \(w\), so it gives the analytic continuation.  The
shifted form \eqref{eq:app-ZD-z-shifted} follows from \(w=z-1/2\) and
\(\Gamma(w-1)/\Gamma(w)=1/(w-1)=1/(z-3/2)\).
\end{proof}

The corresponding energy zeta regularization is
\begin{equation}
\label{eq:app-energy-zeta-def}
    \mathcal E_D(s,a)
    :=
    \frac{\hbarc}{2}\mu^{2s} Z_D\paren{s-\frac12,a},
\end{equation}
where \(\mu\) is an auxiliary reference scale.  The scale \(\mu\) drops out at
\(s=0\) because the continued expression is regular there.  Using
\eqref{eq:app-ZD-continuation},
\begin{align}
\label{eq:app-energy-zeta-evaluated}
    \mathcal E_D(s,a)
    &=
    \frac{\hbarc}{8\pi}
    \mu^{2s}
    \frac{\Gamma(s-3/2)}{\Gamma(s-1/2)}
    \paren{\frac{\pi}{a}}^{3-2s}
    \zeta(2s-3).
\end{align}
At \(s=0\),
\begin{equation}
\label{eq:app-gamma-zeta-values}
    \frac{\Gamma(-3/2)}{\Gamma(-1/2)}=-\frac23,
    \qquad
    \zeta(-3)=\frac{1}{120}.
\end{equation}
Therefore
\begin{equation}
\label{eq:app-zeta-dirichlet-result}
    \mathcal E_D(0,a)
    =
    -\frac{\pi^2\hbarc}{1440a^3}.
\end{equation}

\subsection{Heat asymptotics and the interaction finite part}
\label{subsec:app-heat-asymptotics}

We now extract the same value directly from the heat cutoff.  Insert
\eqref{eq:app-ZD-z-shifted} into the Mellin representation
\eqref{eq:app-mellin-heat-zeta}.  The integrand has the following relevant
poles to the right of any line \(\Re z<0\):
\begin{itemize}
    \item a pole at \(z=2\), from the pole of \(\zeta(2z-3)\);
    \item a pole at \(z=3/2\), from the factor \((z-3/2)^{-1}\);
    \item a pole at \(z=0\), from \(\Gamma(z)\).
\end{itemize}
The first two produce local divergent terms.  The pole at \(z=0\) produces
the finite separation-dependent term.

\begin{proposition}[Short-time expansion of the scalar Dirichlet density]
\label{prop:app-dirichlet-heat-expansion}
As \(\tau\to0^+\), the heat-regulated Dirichlet density has the expansion
\begin{equation}
\label{eq:app-dirichlet-heat-expansion}
    \mathcal E_{D,\tau}(a)
    =
    \frac{\hbarc\,a}{8\pi^2}\tau^{-2}
    -
    \frac{\hbarc}{32\sqrt\pi}\tau^{-3/2}
    -
    \frac{\pi^2\hbarc}{1440a^3}
    +
    O(\tau).
\end{equation}
Consequently, under the standard parallel-plate interaction finite-part
prescription,
\begin{equation}
\label{eq:app-heat-FP-Dirichlet-result}
    \FP_{\mathrm{int},\,\tau\to0^+}\mathcal E_{D,\tau}(a)
    =
    -\frac{\pi^2\hbarc}{1440a^3}.
\end{equation}
\end{proposition}

\begin{proof}
Starting from \eqref{eq:app-mellin-heat-zeta}, shift the vertical contour to
the left.  The decay of \(\Gamma(z)\) in vertical strips and standard polynomial
bounds for \(\zeta\) on vertical lines justify the shift for this explicit
meromorphic integrand.  To obtain the displayed expansion with an \(O(\tau)\)
remainder, one may shift past \(z=0\) and then to any vertical line
\(\Re z<-1\), absorbing the residue at \(z=-1\) and all subsequent lower-order
terms into \(O(\tau)\).  The nonvanishing terms before order \(\tau\) are the
residues of
\[
    \frac{\hbarc}{2}\Gamma(z)\tau^{-z}Z_D\paren{z-\frac12,a}.
\]
At \(z=2\), the pole of \(\zeta(2z-3)\) has residue \(1/2\), and
\eqref{eq:app-ZD-z-shifted} gives
\[
    \operatorname*{Res}_{z=2}
    Z_D\paren{z-\frac12,a}
    =
    \frac{a}{4\pi^2}.
\]
Since \(\Gamma(2)=1\), this contributes
\[
    \frac{\hbarc a}{8\pi^2}\tau^{-2}.
\]
At \(z=3/2\), the residue of \((z-3/2)^{-1}\) is one and \(\zeta(0)=-1/2\), so
\[
    \operatorname*{Res}_{z=3/2}
    Z_D\paren{z-\frac12,a}
    =
    -\frac{1}{8\pi}.
\]
Multiplying by \(\Gamma(3/2)=\sqrt\pi/2\) and by \(\hbarc/2\) gives the term
\[
    -\frac{\hbarc}{32\sqrt\pi}\tau^{-3/2}.
\]
At \(z=0\), the pole is the pole of \(\Gamma(z)\), with residue one.  Therefore
the constant term is
\[
    \frac{\hbarc}{2}Z_D\paren{-\frac12,a}
    =
    \mathcal E_D(0,a)
    =
    -\frac{\pi^2\hbarc}{1440a^3},
\]
using \eqref{eq:app-zeta-dirichlet-result}.

The pole at \(z=2\) is the bulk vacuum divergence proportional to the volume
per unit area, namely to \(a\).  The pole at \(z=3/2\) is the one-plate surface
divergence and is independent of \(a\).  The standard interaction finite part
subtracts these local terms and retains the finite separation-dependent
constant.  This proves \eqref{eq:app-heat-FP-Dirichlet-result}.
\end{proof}

\begin{remark}[What is meant by heat--zeta compatibility]
\label{rem:app-heat-zeta-compatibility}
Equation \eqref{eq:app-heat-FP-Dirichlet-result} is not a claim that arbitrary
finite counterterms are scheme-independent.  It says that for the standard
parallel-plate interaction prescription, the finite term extracted from the
heat cutoff \(\lambda^{1/2}e^{-\tau\lambda}\) agrees with the analytic
continuation value \(\mathcal E_D(0,a)\).  This is the compatibility used in
Section~\ref{sec:finite-part-plate-result}.
\end{remark}

\subsection{Reduced Neumann branch}
\label{subsec:app-reduced-neumann-branch}

The scalar Neumann comparison operator with its constant zero mode removed has
vertical modes \(n\ge1\), matching the Dirichlet eigenvalues, plus the
zero-vertical branch \(n=0\), \(m\ne0\).  At the heat-regulated finite-area
level this is exactly the identity
\[
    \ENtau(L,a)=\EDtau(L,a)+\Btau(L),
\]
with \(\Btau(L)\) defined in \eqref{eq:Btau-def-section06}.  The large-area
limit of this branch is
\begin{align}
\label{eq:app-Btau-large-area}
    \lim_{L\to\infty}\frac{\Btau(L)}{L^2}
    &=
    \frac{\hbarc}{2}
    \int_{\R^2}\frac{\dd^2k}{(2\pi)^2}\abs{k}e^{-\tau\abs{k}^2}  \\
    &=
    \frac{\hbarc}{16\sqrt\pi}\tau^{-3/2}.
\end{align}
It has no dependence on the plate separation \(a\).  Hence
\begin{equation}
\label{eq:app-Btau-FP-zero}
    \FP_{\mathrm{int},\,\tau\to0^+}
    \left[
        \lim_{L\to\infty}\frac{\Btau(L)}{L^2}
    \right]
    =0.
\end{equation}
Therefore the reduced Neumann scalar channel has the same interaction finite
part as the Dirichlet scalar channel:
\begin{equation}
\label{eq:app-Neumann-FP-equals-Dirichlet}
    \FP_{\mathrm{int},\,\tau\to0^+}
    \left[
        \lim_{L\to\infty}\frac{\ENtau(L,a)}{L^2}
    \right]
    =
    -\frac{\pi^2\hbarc}{1440a^3}.
\end{equation}
This is the only Neumann input needed for the Maxwell plate coefficient: the
Maxwell trace is \(\LD\oplus\LNp\), and \(\LNp\) differs from \(\LD\) only by
an \(a\)-independent branch at the interaction-energy level.

\begin{remark}[No additional physical assumption]
\label{rem:app-no-extra-assumption}
The subtraction of \eqref{eq:app-Btau-large-area} is not a new assumption about
Maxwell theory.  It is the already stated parallel-plate interaction finite
part: terms independent of the plate separation do not contribute to the
interaction energy or force.  The branch is retained in the regulated trace and
removed only at the finite-part stage.
\end{remark}
\section{Gauge-potential formulations and the physical electric-field trace}
\label{app:optional-gauge-comment}

The main text deliberately avoids using a vector potential as the primary
Hilbert-space variable.  Instead it defines the reduced Maxwell operator
\(\LMx\) directly on divergence-free electric fields satisfying the perfect
conductor tangential boundary condition, with the static normal zero mode
removed.  This appendix explains how that choice relates to the more familiar
vector-potential language for the nonzero parallel-plate modes.  Nothing in the
main proof depends on this appendix.

There are two reasons for keeping the main proof in the electric-field
formulation.  First, the physical perfect-conductor boundary condition is
stated directly on the fields:
\begin{equation}
    n\times E=0,
    \qquad
    n\cdot B=0
    \quad \text{on the plates}.
    \label{eq:PEC-field-conditions-appendixC}
\end{equation}
Second, the vector potential has a gauge redundancy.  If one uses \(A\) as the
basic variable, one must either quotient by gauge transformations, impose a
gauge condition and handle residual gauge freedom, or work in a covariant
gauge-fixed formalism with the associated ghost determinant.  The electric-field
construction in Sections~\ref{sec:maxwell-operator}--\ref{sec:te-tm-decomposition}
works directly with the reduced physical Hilbert space and therefore avoids
introducing gauge variables that later have to be removed.

\subsection{Nonzero modes and temporal/Coulomb gauge}
\label{subsec:nonzero-modes-temporal-coulomb}

We use the time dependence \(\ee^{-\ii\omega t}\).  For a nonzero-frequency
source-free Maxwell mode in vacuum, Faraday's law gives
\begin{equation}
    B=\frac{1}{\ii\omega}\,\nabla\times E,
    \label{eq:B-from-E-appendixC}
\end{equation}
and Ampere's law gives
\begin{equation}
    \nabla\times\nabla\times E
    =\frac{\omega^2}{c^2}E,
    \qquad
    \nabla\cdot E=0 .
    \label{eq:E-curlcurl-nonzero-appendixC}
\end{equation}
Thus the eigenvalue \(\lambda\) of the spatial operator \(\LMx\) is related to
the physical angular frequency by
\begin{equation}
    \omega=c\sqrt{\lambda}.
    \label{eq:omega-lambda-appendixC}
\end{equation}
For the operator-theoretic parts of the paper, the factor \(c\) is kept outside
the spectral operator and appears in the trace as \(\hbar c\sqrt{\lambda}/2\).

For nonzero modes one may choose temporal gauge, \(\Phi=0\), so that
\begin{equation}
    E=-\partial_t A=\ii\omega A.
    \label{eq:E-iomega-A-appendixC}
\end{equation}
Since \(\omega\neq0\), this gives a one-to-one correspondence between electric
field modes and vector-potential modes:
\begin{equation}
    A=\frac{1}{\ii\omega}E .
    \label{eq:A-from-E-appendixC}
\end{equation}
If \(\nabla\cdot E=0\), then \(\nabla\cdot A=0\).  Thus the corresponding
potential is in Coulomb gauge.  Moreover, the perfect-conductor condition
\(n\times E=0\) implies
\begin{equation}
    n\times A=0
    \quad \text{on the plates}.
    \label{eq:ncrossA-appendixC}
\end{equation}
Conversely, a nonzero-frequency Coulomb-gauge vector-potential mode satisfying
\(n\times A=0\) yields a physical electric mode by \(E=\ii\omega A\).

The magnetic boundary condition in \eqref{eq:PEC-field-conditions-appendixC}
is also compatible with this description.  Indeed, on the flat plates
\(z=0,a\), if the tangential trace of \(A\) vanishes, then the normal component
of \(\nabla\times A\) is
\begin{equation}
    n\cdot(\nabla\times A)
    =\partial_x A_y-\partial_y A_x,
    \label{eq:normal-curl-A-appendixC}
\end{equation}
with the right-hand side understood in the tangential trace sense.  Since the
tangential components of \(A\) vanish on the flat boundary, their tangential
derivatives vanish there in the trace sense for smooth modes, and then by
density for the form-domain modes.  Hence
\begin{equation}
    n\cdot B=n\cdot(\nabla\times A)=0.
    \label{eq:n-dot-B-from-A-appendixC}
\end{equation}
This is the field-level reason why the boundary condition used in the main
text captures the perfect-conductor plate problem.

\begin{proposition}[Mode-level equivalence for nonzero finite-volume modes]
\label{prop:mode-level-equivalence-A-E}
For the finite slab \(\OmegaLa=T_L^2\times[0,a]\), the nonzero eigenmodes of
\(\LMx\) on the reduced electric-field Hilbert space are in one-to-one
correspondence with nonzero-frequency Coulomb-gauge vector-potential modes
satisfying \(n\times A=0\) on the plates.  Under this correspondence the
spatial eigenvalue is unchanged.
\end{proposition}

\begin{proof}
Let \(E\in\Hphys\) be an eigenmode of \(\LMx\) with eigenvalue
\(\lambda>0\).  Set \(\omega=c\sqrt{\lambda}\) and define
\(A=(\ii\omega)^{-1}E\).  Then \(\nabla\cdot A=0\), since
\(\nabla\cdot E=0\), and \(n\times A=0\), since \(n\times E=0\).  Also
\begin{equation}
    \nabla\times\nabla\times A
    =\lambda A,
\end{equation}
because the equation is obtained from the corresponding eigenvalue equation
for \(E\) by multiplication by the nonzero scalar \((\ii\omega)^{-1}\).  Thus
\(A\) is a Coulomb-gauge vector-potential mode with the same spatial
eigenvalue.

Conversely, let \(A\) be a Coulomb-gauge vector-potential mode satisfying
\(n\times A=0\) and
\(\nabla\times\nabla\times A=\lambda A\) with \(\lambda>0\).  With
\(\omega=c\sqrt{\lambda}\), define \(E=\ii\omega A\).  Then
\(\nabla\cdot E=0\), \(n\times E=0\), and
\(\nabla\times\nabla\times E=\lambda E\).  Hence \(E\) is a nonzero eigenmode
of \(\LMx\) in the physical electric-field space.  These two constructions are
inverse to each other.
\end{proof}

\begin{remark}[The removed static mode]
\label{rem:static-mode-no-A-oscillator}
The mode \(E=h_0\), removed in \eqref{eq:Hphys-def}, has eigenvalue zero and
represents a uniform static normal electric field.  It is not a nonzero photon
oscillator and is not covered by the temporal-gauge correspondence above.  This
is one reason the finite-volume construction first removes \(h_0\) and only
then forms \(\LMx^{-1}\).
\end{remark}

\subsection{Relation to the TE/TM trace}
\label{subsec:gauge-comment-te-tm}

The TE/TM decomposition in Section~\ref{sec:te-tm-decomposition} is not a
gauge artifact.  It is a decomposition of the reduced physical Maxwell
Hilbert space itself.  For \(m\ne0\) and \(n\ge1\), the two modes
\(E^{\mathrm{TE}}_{m,n}\) and \(E^{\mathrm{TM}}_{m,n}\) form an orthonormal
basis of the two-dimensional divergence-free coefficient block
\(W_{m,n}\) in \eqref{eq:Wmn-section05}.  For \(m=0\), the TE/TM labels are
not canonical, but the physical multiplicity is still two.  For \(n=0\), the
normal branch contributes one reduced Neumann-type family, with the constant
zero mode removed.

Thus the trace identity
\begin{equation}
    \Trphys f(\LMx)
    =\Tr f(\LD)+\Tr' f(\LN)
    \label{eq:gauge-comment-DN-trace}
\end{equation}
from Theorem~\ref{thm:maxwell-trace-Dirichlet-Neumann} is a statement about
physical Maxwell modes.  The factor of two comes from the physical Maxwell
mode multiplicities after the unphysical degrees of freedom have been removed.

\subsection{What a covariant gauge-fixed calculation would have to show}
\label{subsec:covariant-gauge-fixed-comment}

A covariant path-integral treatment would start from a gauge-fixed action for
a vector potential and would include ghost fields.  With boundaries, the gauge
condition, the vector boundary conditions, and the ghost boundary conditions
must be chosen compatibly.  That is a different formulation from the one used
in this paper.

For the present flat parallel-plate geometry, such a calculation would have to
recover the same reduced physical trace,
\begin{equation}
    \frac{\hbar c}{2}\Trphys\paren{\LMx^{1/2}\ee^{-\tau\LMx}},
    \label{eq:expected-covariant-endpoint-appendixC}
\end{equation}
because after gauge and ghost cancellations only the physical TE and TM
oscillators remain.  However, the main text does not need this determinant
identity and does not prove it in general.  Instead, it constructs the physical
Hilbert space first and proves the trace identity directly from the Maxwell
mode decomposition.

This distinction is important.  The statement proved in this paper is a
Riesz/Gaussian representation on the reduced physical Maxwell space, followed
by the standard parallel-plate interaction finite part.  Questions of
covariant path-integral derivation and of other boundary geometries lie beyond
the scope of this appendix.

\subsection{Differential-form terminology}
\label{subsec:differential-form-comment}

In differential-form language, the electric field is a one-form and the
condition \(n\times E=0\) on a plate is the vanishing of the tangential
pullback of that one-form.  This is the relative trace condition for
one-forms.  The corresponding operator-domain boundary conditions are encoded
here through the closed curl--curl form rather than imposed separately.  The
constraint \(\nabla\cdot E=0\) is the co-closedness condition.
Thus the main Hilbert space may equivalently be described as the reduced space
of co-closed relative one-forms, with the harmonic relative one-form
corresponding to the constant normal electric field removed.

This form-language description is useful for comparing with other treatments
of Maxwell theory with boundary.  It should not obscure the concrete content
of the present proof: in the slab geometry the physical operator is the
positive self-adjoint curl-curl operator defined by the quadratic form
\(\qMx[E]=\norm{\nabla\times E}_{L^2}^2\) on
\(\Vphys\), and its finite-volume spectrum is the one computed explicitly in
Sections~\ref{sec:fourier-domain-heat-trace} and~\ref{sec:te-tm-decomposition}.

\begin{remark}[No additional input to the main theorem]
\label{rem:gauge-comment-no-additional-input}
The discussion in this appendix supplies interpretation only.  The proof of
the Maxwell Riesz reduction, the Gaussian quadratic-form identity, and the
parallel-plate finite part uses the electric-field operator \(\LMx\), the
finite-volume spectral gap, the heat-trace bounds, and the TE/TM trace identity
proved in the main text.  It does not use a gauge-fixed determinant, a ghost
operator, or a vector-potential path integral.
\end{remark}

\end{document}